\documentclass[12pt]{amsart}
\usepackage{amsmath,latexsym,amscd,amsbsy,amssymb,amsfonts,amsthm,fleqn,leqno,
euscript, graphicx, texdraw, pb-diagram}
\usepackage{mathrsfs}
\usepackage{color}
\theoremstyle{plain}
\newtheorem{thm}{Theorem}[section]
\newtheorem{lem}[thm]{Lemma}

\newtheorem{cor}[thm]{Corollary}
\newtheorem{pro}[thm]{Proposition}

\newtheorem{claim}{Claim}

\begin{document}
\title{On uniformly continuous surjections between function spaces}

\author{Al\.{ı} Emre Eysen}
\address{Department of Mathematics, Faculty of Science, Trakya University, Edirne, Turkey}
\email{aemreeysen@trakya.edu.tr}
\thanks{The first author was partially supported by TUBITAK-2219}

\author[c]{Vesko Valov}
\address{Department of Computer Science and Mathematics, Nipissing University,
100 College Drive, P.O. Box 5002, North Bay, ON, P1B 8L7, Canada}
\email{veskov@nipissingu.ca}
\thanks{The second author was partially supported by NSERC Grant 261914-19}



\begin{abstract}
We consider uniformly continuous surjections between $C_p(X)$ and $C_p(Y)$ (resp, $C_p^*(X)$ and $C_p^*(Y$)) and show that if $X$ has some dimensional-like properties, then so does $Y$. In particular, we prove that if $T:C_p(X)\to C_p(Y)$ is a continuous linear surjection and $\dim X=0$, then $\dim Y=0$. This provides a positive answer to a question raised by Kawamura-Leiderman \cite[Problem 3.1]{kl}.    
\end{abstract}
\maketitle

\markboth{}{Uniformly continuous surjections}



\section{Introduction}
All spaces in this paper, if not said otherwise, are Tychonoff spaces and all maps are continuous.
By $C(X)$ (resp., $C^*(X)$) we denote the set of all continuous (resp., continuous and bounded) real-valued functions on a space $X$. We write 
$C_p(X)$ (resp., $C_p^*(X)$) for the spaces $C(X)$ (resp., $C^*(X)$) endowed with the pointwise topology. More information about the spaces $C_p(X)$ can be found in \cite{tk}.
By dimension we mean the {\em covering dimension $\dim$} defined by finite functionally open covers, see \cite{en}. According to that definition, we have $\dim X=\dim\beta X$, where $\beta X$ is the \v{C}ech-Stone compactification of $X$. 

After the striking result of Pestov \cite{p} that $\dim X=\dim Y$ provided that $C_p(X)$ and $C_p(Y)$ are linearly homeomorphic, and Gul'ko's \cite{gu} generalization of Pestov's theorem that the same is true for uniformly continuous homeomorphisms, a question arose whether $\dim Y\leq\dim X$ if there is continuous linear surjection from $C_p(X)$ onto $C_p(Y)$, see \cite{ar}. This was answered negatively by Leiderman-Levin-Pestov \cite{llp} and Leiderman-Morris-Pestov \cite{lmp}. On the other hand, it was shown in \cite{llp} that if there is a linear continuous surjection $C_p(X)\to C_p(Y)$ such that $X$ and $Y$ are compact metrizable spaces and $\dim X=0$, then $\dim Y=0$. The last result was extended for arbitrary compact spaces by Kawamura-Leiderman \cite{kl} who also raised the question whether this is true for any Tychonoff spaces $X$ and $Y$. In this paper we provide a positive answer to that question and discuss the situation when the surjection $T:C_p(X)\to C_p(Y)$ is uniformly continuous. Let us note that the preservation of dimension under linear homeomorphisms doesn't hold for function spaces with the uniform norm topology. Indeed, according to the classical result of Milutin \cite{mi} if $X$ and $Y$ are any uncountable metrizable compacta, then there is a linear homeomorphism between the Banach spaces $C(X)$ and $C(Y)$.  

Suppose $E_p(X)\subset C_p(X)$ and $E_p(Y)\subset C_p(Y)$ are subspaces containing the zero functions on $X$ and $Y$, respectively. Recall that a map $\varphi: E_p(X)\to E_p(Y)$ is {\em uniformly continuous} if for every neighborhood $U$ of the zero function in $E_p(Y)$ there is a neighborhood $V$ of the zero function in $E_p(X)$ such that $f,g\in E_p(X)$ and $f-g\in V$ implies $\varphi(f)-\varphi(g)\in U$. Evidently, if $E_p(X)$ and $E_p(Y)$ are linear spaces, then every linear continuous map between $E_p(X)$ and $E_p(Y)$ is uniformly continuous.
If $f\in C_p(X)$ is a bounded function, then $||f||$ stands for the supremum norm of $f$. The notion of $c$-good maps was introduced in \cite{gkm} (see also \cite{gf}), where $c$ is a positive number. 
A map $\varphi: E_p(X)\to E_p(Y)$ is {\em $c$-good} if for every bounded function $g\in E_p(Y)$ there exists a bounded function $f\in E_p(X)$ such that $\varphi(f)=g$ and $||f||\leq c||g||$.

Everywhere below, by $D(X)$ we denote either $C^*(X)$ or $C(X)$. Here is one of our main results.
\begin{thm}
Let $T:D_p(X)\to D_p(Y)$ be a $c$-good uniformly continuous surjection for some $c>0$.
Then $Y$ is $0$-dimensional provided so is $X$.
\end{thm} 

\begin{cor}
Suppose there is a linear continuous surjection from $C_p^*(X)$ onto $C_p^*(Y)$.
Then $Y$ is $0$-dimensional provided that so is $X$.
\end{cor}
Corollary 1.2 follows from Theorem 1.1 because every linear continuous surjection between $C_p^*(X)$ and $C_p^*(Y)$ is $c$-good for some $c>0$, see Proposition 3.3. Moreover, as one of the referees observed, it is not possible to have a continuous linear surjection $T:C_p^*(X)\to C_p(Y)$ with $C_p(Y)\neq C_p^*(Y)$. Indeed, we consider the composition $T\circ i$, where $i:C_p(\beta X)\to C_p^*(X)$ is the restriction map. Then, according to \cite[Proposition 2]{vu}, $Y$ is pseudocompact and $C_p(Y)=C_p^*(Y)$.  

We consider properties $\mathcal P$ of $\sigma$-compact metrizable spaces such that:
\begin{itemize}
\item [(a)] If $X\in\mathcal P$ and $F\subset X$ is closed, then $F\in\mathcal P$;
\item [(b)] If $X$ is a countable union of closed subsets each having the property $\mathcal P$, then $X\in\mathcal P$;
\item [(c)] If $f:X\to Y$ is a perfect map with countable fibers and $Y\in\mathcal P$, then $X\in\mathcal P$.
\end{itemize}
For example, {\em $0$-dimensionality}, {\em strongly countable-dimensionality} and {\em $C$-space property} are finitely multiplicative properties satisfying conditions $(a)-(c)$. Another two properties of this type, but not finitely multiplicative, 
are {\em weakly infinite-dimensionality}, or {\em $(m-C)$-spaces} in the sense of Fedorchuk \cite{vf}. 
The definition of all dimension-like properties mentioned above, and explanations that they satisfy conditions $(a)-(c)$ can be found at the end of Section 2.

For $\sigma$-compact metrizable spaces we have the following version of Theorem $1.1$ (actually Theorem 1.3 below is true for any topological property $\mathcal P$ satisfying conditions $(a)-(c)$):
\begin{thm}
Suppose $X$ and $Y$ are $\sigma$-compact metrizable spaces and let $T:D_p(X)\to D_p(Y)$ be a $c$-good uniformly continuous surjection for some $c>0$. 
If all finite powers of $X$ are weakly infinite-dimensional or $(m-C)$-spaces, then all finite powers of $Y$ are also weakly infinite-dimensional or $(m-C)$-spaces.
\end{thm}
M. Krupski \cite{k1} proved similar result for $\sigma$-compact metrizable spaces: if $T:C_p(X)\to C_p(Y)$ is a continuous open surjection and all powers of $X$ are weakly infinite-dimensional or $(m-C)$-spaces, then $Y$ is also weakly infinite-dimensional of has the property $m-C$.
Theorem 1.3 was established in \cite{gkm} in the special case when $X,Y$ are compact metrizable spaces and the property is either $0$-dimensionality or strong countable-dimensionality. 

The notion of $c$-good maps is crucial in the proof of the above results. One of the referees observed that if $X=Y$ is the space of natural numbers, then the continuous linear surjection $T:C_p(X)\to C_p(Y)$, $T(f)(n)=\frac{1}{n+1}f(n)$, is not $c$-good for any $c$. Therefore, as the referee suggested, the more interesting question is whether the existence of a linear continuous surjection $T:C_p(X)\to C_p(Y)$ implies the existence of another linear continuous surjection $S:C_p(X)\to C_p(Y)$ which is $c$-good for some $c>0$. In that connection, let us note that more general notion 
was considered in our recent paper \cite{elv}: $T:D_p(X)\to D_p(Y)$ is {\em inversely bounded} if for every norm bounded sequence $\{g_n\}\subset C^*(Y)$ there exists a norm bounded sequence $\{f_n\}\subset C^*(X)$ with $T(f_n)=g_n$ for all $n$. Evidently, every $c$-good map is inversely bounded. It was established in \cite{elv} that uniformly continuous inversely bounded surjections $T:D_p(X)\to D_p(Y)$, where $X$ and $Y$ are metrizable, preserve any one of the properties $0$-dimensionality, countable-dimensionality or strong countable-dimensionality. 
Most probably Theorem 1.1 remains true if the surjection $T$ is uniformly continuous and inversely bounded.

Finally, here is the theorem which provides a positive answer to the question of Kawamura-Leiderman \cite[Problem 3.1]{kl} mentioned above:
\begin{thm}
Let $T:C_p(X)\to C_p(Y)$ be a linear continuous surjection. If $\dim X=0$, then $\dim Y=0$. 
\end{thm}
\section{Preliminary results}
In this section we prove Proposition 2.1 which is used in the proofs of Theorem $1.1$ and Theorem $1.3$.  Our proof is based on the idea of support introduced by Gul'ko \cite{gu} and the extension of this notion introduced by Krupski \cite{k}.
 
Let $\mathbb Q$ be the set of rational numbers. A subspace $E(X)\subset C(X)$ is called a {\em $QS$-algebra} \cite{gu} if it satisfies the following conditions: (i) If $f,g\in E(X)$ and $\lambda\in\mathbb Q$, then all functions $f+g$, $f\cdot g$ and $\lambda f$ belong to $E(X)$; (ii) For every $x\in X$ and its neighborhood $U$ in $X$ there is $f\in E(X)$ such that $f(x)=1$ and $f(X\backslash U)=0$. 

We are using the following facts from \cite{gu}: 
\begin{itemize}
\item[(2.1)] If $X$ has a countable base and $\Phi\subset C(X)$ is a countable set, then there is a countable $QS$-algebra $E(X)\subset C(X)$ containing $\Phi$. Moreover, it follows from the proof of \cite[Proposition 1.2]{gu} that
    $E(X)\subset C^*(X)$ provided that $\Phi\subset C^*(X)$; 
\item[(2.2)] If $U$ is an open set in $X$, $x_1,x_2,..,x_k\in U$ and $\lambda_1,\lambda_2,..,\lambda_k\in\mathbb Q$, then there exists $f\in E(X)$ such that $f(x_i)=\lambda_i$ for each $i$ and $f(X\backslash U)=0$. 
\item[(2.3)] We consider the following condition for a $QS$-algebra $E(X)$ on $X$: For every compact set  $K\subset X$ and an open set $W$ containing $K$ there exists $f\in E(X)$ with $f|K=1$, $f|(X\backslash W)=0$ and $f(x)\in [0,1]$ for all $x\in X$. Note that 
if $X$ has a countable base $\mathcal B$, then there is a countable $QS$-algebra $E(X)$ on $X$ satisfying that condition. Indeed, we can assume that $\mathcal B$ is closed under finite unions and find $U,V\in\mathcal B$ such that $K\subset V\subset\overline V\subset U\subset\overline U\subset W$. Then consider the set $\Phi$ of all functions $f_{U,V}:X\to [0,1]$, where $\overline V\subset U$ with $U,V\in\mathcal B$, such that  
    $f_{V,U}|\overline V=1$ and $f_{V,U}|(X\backslash U)=0$. According to $(2.1)$, $\Phi$ can be extended to a countable $QS$-algebra $E(X)$ on $X$. 
\end{itemize}

Everywhere below we denote by $\overline{\mathbb R}$ the extended real line $[-\infty,\infty]$.
\begin{pro}
Let $\overline X$ and $\overline Y$ be metrizable compactifications of $X$ and $Y$, and $H\subset\overline X$ be a $\sigma$-compact space containing $X$. Suppose $E(H)$ is a $QS$-algebra on $H$ satisfying condition $(2.3)$, $E(X)=\{\overline f|X:\overline f\in E(H)\}$ and $E(Y)\subset C(Y)$ is a family such that every 
$g\in E(Y)$ is extendable to a map $\overline g:\overline Y\to\overline{\mathbb R}$ and $E(\overline Y)=\{\overline g:g\in E(Y)\}$ 
contains a $QS$-algebra on $\overline Y$. 
Let also $\varphi:E_p(X)\to E_p(Y)$ be a uniformly continuous surjection  which is $c$-good for some $c>0$. 

If all finite powers  of $H$ have a property $\mathcal P$ satisfying conditions $(a)-(c)$, then there exists a $\sigma$-compact set $Y_\infty\subset\overline Y$ containing $Y$ such that all finite powers of $Y_\infty$ have the same property $\mathcal P$. 
\end{pro}
\begin{proof}
We fix a countable base $\mathcal B$ of $H$ which is closed under finite unions, and denote by  $f$ the restriction $\overline f|X$ of any 
$\overline f\in E(H)$. 
For every $y\in\overline Y$ there is a map $\alpha_y:E(H)\to\overline{\mathbb R}$, $\alpha_y(\overline f)=\overline{\varphi(f)}(y)$. Since $\varphi$ is uniformly continuous, so is the map $\beta_y:E_p(X)\to\mathbb R$, $\beta_y(f)=\varphi(f)(y)$. 
Let $H=\bigcup_kH_k$ be the union of an increasing sequence $\{H_k\}$ of compact sets.
Following Krupski \cite{k}, for every $y\in\overline Y$ and every $p,k\in\mathbb N$ we define the families
$$\mathcal A^k(y)=\{K\subset H_k:K{~}\mbox{is closed and}{~}a(y,K)<\infty\}$$ and
$$\mathcal A_p^k(y)=\{K\subset H_k:K{~}\mbox{is closed and}{~}a(y,K)\leq p\},$$
where 
$$a(y,K)=\sup\{|\alpha_y(\overline f)-\alpha_y(\overline g)|:\overline f,\overline g\in E(H), |\overline f(x)-\overline g(x)|<1{~}\forall x\in K\}.$$
Possibly, some or both of the values 
$\alpha_y(\overline f),\alpha_y(\overline g)$ from the definition of $a(y,K)$ could be $\pm\infty$. That's why we use the following agreements:
\begin{itemize}
\item[(2.4)] $\infty+\infty=\infty, \infty-\infty=-\infty+\infty=0, -\infty-\infty=-\infty.$
\end{itemize}
Note that $a(y,\varnothing)=\infty$ since $\varphi$ is surjective.

Using that $E(X)$ and $E(H)$ are $QS$-algebras on $X$ and $H$, and following the arguments from Krupski's paper \cite{k} (see also the proofs of \cite[Proposition 1.4]{gu} and \cite[Proposition 3.1]{mp}), one can establish the following claims (for the sake of completeness we provide the proofs):
\begin{claim}
For every $y\in Y$ there is $p,k\in\mathbb N$ such that $\mathcal A_p^k(y)$ contains a finite nonempty subset of $X$.
\end{claim}
This claim follows from the proof of \cite[Proposition 2.1]{k}. Indeed, since $\varphi$ is uniformly continuous there is $p\in\mathbb N$ and a finite set $K\subset X$ such that if $f,g\in E(X)$ and $|f(x)-g(x)|<1/p$ for every $x\in K$, then 
$|\alpha_y(\overline f)-\alpha_y(\overline g)|=|\varphi(f)(y)-\varphi(g)(y)|<1$. Take arbitrary $\overline f,\overline g\in E(H)$ with $|\overline f(x)-\overline g(x)|<1$ for every $x\in K$ and consider the functions
$\overline f_m=\overline f+ \frac{m}{p}(\overline g-\overline f)\in E(H)$ for each $m=0,1,..,p$. Obviously $|f_m(x)-f_{m+1}(x)|<1/p$ for all $x\in K$, so 
$|\alpha_y(\overline f_m)-\alpha_y(\overline f_{m+1})|<1$. Consequently, $|\alpha_y(\overline f)-\alpha_y(\overline g)|\leq\sum_{m=0}^{p-1}|\alpha_y(\overline f_m)-\alpha_y(\overline f_{m+1})|<p$. Because $K$ is finite, there is $k\in\mathbb N$ with $K\subset H_k$. Hence $K\in\mathcal A_p^k(y)$.

Consider the sets $Y_p^k=\{y\in\overline Y:\mathcal A_p^k(y)\neq\varnothing\}$, $p,k\in\mathbb N$.
\begin{claim}
Each $Y_p^k$ is a closed subset of $\overline Y$.
\end{claim}
We use the proof of \cite[Lemma 2.2]{k}. Suppose $y\not\in Y_p^k$. Since $y\in Y_p^k$ iff $H_k\in\mathcal A_p^k(y)$, $H_k\not\in\mathcal A_p^k(y)$. So, there exist $\overline f,\overline g\in E(H)$ with $|\overline f(x)-\overline g(x)|<1$ for all $x\in H_k$ and 
$|\alpha_y(\overline f)-\alpha_y(\overline g)|>p$. Then $V=\{z\in\overline Y: |\alpha_z(\overline f)-\alpha_z(\overline g)|>p\}$ is a neighborhood of $y$ with $V\cap Y_p^k=\varnothing$. 

\begin{claim}
Every set $Y_{p,q}^k=\{y\in Y_p^k: \exists K\in\mathcal A_p^k(y){~}\mbox{with}{~}|K|\leq q\}$, $p,q,k\in\mathbb N$, is closed in $Y_p^k$.
\end{claim}
Following the proof of \cite[Lemma 2.3]{k}, we first show that the set $Z=\{(y,K)\in Y_p^k\times [H_k]^{\leq q}:K\in\mathcal A_p^k(y)\}$ is closed in 
$Y_p^k\times [H_k]^{\leq q}$, where $[H_k]^{\leq q}$ denotes the space of all subsets $K\subset H_k$ of cardinality $\leq q$ endowed with the Vietoris topology. Indeed, if $(y,K)\in Y_p^k\times [H_k]^{\leq q}\backslash Z$, then $K\not\in\mathcal A_p^k(y)$. Hence, $a(y,K)>p$ and there are $\overline f,\overline g\in E(H)$ such that $|\overline f(x)-\overline g(x)|<1$ for all $x\in K$ and $|\alpha_y(\overline f)-\alpha_y(\overline g)|>p$.
Let $U=\{z\in Y_p^k: |\alpha_z(\overline f)-\alpha_z(\overline g)|>p\}$ and $V=\{x\in H_k:|\overline f(x)-\overline g(x)|<1\}$. The set $U\times<V>$ is a neighborhood of $(y,K)$ in $Y_p^k\times [H_k]^{\leq q}$ disjoint from $Z$ (here $<V>=\{F\in [H_k]^{\leq q}: F\subset V\}$).
Since $Y_p^k\times [H_k]^{\leq q}$ is compact and $Y_{p,q}^k$ is the image of $Z$ under the projection $Y_p^k\times [H_k]^{\leq q}\to Y_p^k$, $Y_{p,q}^k$ is closed in $Y_p^k$.

For every $k$ let $Y_k=\bigcup_{p,q}Y_{p,q}^k$. Obviously, $Y_k\subset\{y\in\overline Y:\mathcal A^k(y)\neq\varnothing\}$. Since $H_k\subset H_{k+1}$ for all $k$, the sequence $\{Y_k\}$ is increasing. 
It may happen that $Y_k=\varnothing$ for some $k$, but Claim 1 implies that $Y\subset\bigcup_{k}Y_{k}$. 
\begin{claim}
For every $y\in Y_k$ the family $\mathcal A^k(y)$ is closed under finite intersections and $a(y,K_1\cap K_2)\leq a(y,K_1)+a(y,K_2)$ for all $K_1,K_2\in\mathcal A^k(y)$.
\end{claim} 
We follow the proof of \cite[Lemma 2.5]{k} to show that $K_1\cap K_2\in\mathcal A^k(y)$ for any $K_1,K_2\in\mathcal A^k(y)$. Let $\overline f,\overline g\in E(H)$ with $|\overline f(x)-\overline g(x)|<1$ for all $x\in K_1\cap K_2$ and $U=\{x\in H:|\overline f(x)-\overline g(x)|<1\}$. Take an open set $W$ in $H$ containing $K_1$ with $W\cap K_2\subset U$ and choose $\overline u\in E(H)$ such that $\overline u|K_1=1$, 
$\overline u|(H\backslash W)=0$ and $\overline u(x)\in [0,1]$ for all $x\in H$, see condition $(2.3)$. Then $\overline h=\overline u\cdot(\overline f-\overline g)+\overline g\in E(H)$, $\overline h|K_1=\overline f|K_1$, $\overline h|(K_2\backslash W)=\overline g|(K_2\backslash W)$ and $|\overline h(x)-\overline g(x)|<1$ for $x\in K_2$.
Since $K_1\in\mathcal A^k(y)$ and $\overline h|K_1=\overline f|K_1$, we have $|\alpha_y(\overline f)-\alpha_y(\overline h)|\leq a(y,K_1)<\infty$. Similarly, $K_2\in\mathcal A^k(y)$ and $|\overline h(x)-\overline g(x)|<1$ for $x\in K_2$ imply $|\alpha_y(\overline h)-\alpha_y(\overline g)|\leq a(y,K_2)<\infty$.
Therefore, $$|\alpha_y(\overline f)-\alpha_y(\overline g)|\leq |\alpha_y(\overline f)-\alpha_y(\overline h)|+|\alpha_y(\overline h)-\alpha_y(\overline g)|\leq a(y,K_1)+a(y,K_2).$$ 
Note that the last inequality is true if some of $\alpha_y(\overline f), \alpha_y(\overline h), \alpha_y(\overline g)$ are $\pm\infty$. Indeed, if 
$\alpha_y(\overline f)=\pm\infty$, then $|\alpha_y(\overline f)-\alpha_y(\overline h)|<\infty$ implies $\alpha_y(\overline h)=\pm\infty$. Consequently,
$\alpha_y(\overline g)=\pm\infty$ because $|\alpha_y(\overline h)-\alpha_y(\overline g)|<\infty$. Similarly, if $\alpha_y(\overline h)=\pm\infty$ or 
$\alpha_y(\overline g)=\pm\infty$, then the other two are also $\pm\infty$. 
Hence, $a(y,K_1\cap K_2)\leq a(y,K_1)+a(y,K_2)$, which  means that $K_1\cap K_2\neq\varnothing$ (otherwise $a(y,K_1\cap K_2)=\infty$) and $K_1\cap K_2\in\mathcal A^k(y)$.

Since each family $\mathcal A^k(y)$, $y\in Y_k$, consists of compact subsets of $H_k$, $K(y,k)=\bigcap\mathcal A^k(y)$ is nonempty and compact. 
\begin{claim}
For every $y\in Y_{k}$ the set $K(y,k)$ is a nonempty finite subset of $H_k$ with $K(y,k)\in\mathcal A^k(y)$. Moreover, if $y\in Y$ then there exists $k$ such that $y\in Y_k$ and $K(y,k)\subset X$.
\end{claim}
Let $y\in Y_k$. We already observed that $K(y,k)$ is compact and nonempty. Since $y\in Y_{p,q}^k$ for some $p,q$, $\mathcal A^k(y)$ contains finite sets. Hence,
$K(y,k)$ is also finite and  $K(y,k)\in\mathcal A^k(y)$ because it is an intersection of finitely many elements of $\mathcal A^k(y)$. If $y\in Y$, then by Claim 1, there is $k$ such that $\mathcal A^k(y)$ contains a finite subset of $X$. Since $K(y,k)$ is the minimal element of $\mathcal A^k(y)$, it is also a subset of $X$.

Following \cite{gu}, for every $k$ we define $M^k(p,1)=Y_{p,1}^k$ and $M^k(p,q)=Y_{p,q}^k\backslash Y_{2p,q-1}^k$ if $q\geq 2$. 
\begin{claim}
$Y_k=\bigcup\{M^k(p,q):p,q=1,2,..\}$ and for every $y\in M^k(p,q)$ there exists a unique set $K_{kp}(y)\in\mathcal A^k(y)$ of cardinality $q$ such that $a(y,K_{kp}(y))\leq p$. 
\end{claim}
Since $M^k(p,q)\subset Y^k_{p,q}\subset Y_k$, $\bigcup\{M^k(p,q):p,q=1,2,..\}\subset Y_k$.
If $y\in Y_k$, then $K(y,k)\in\mathcal A^k(y)$ is a finite subset of $H_k$. Assume $|K(y,k)|=q$ and $a(y,K(y,k))\leq p$ for some $p,q$. So, $y\in Y_{p,q}^k$. Moreover $y\not\in Y_{2p,q-1}^k$, otherwise there would be $K\in\mathcal A^k(y)$ with $a(y,K)\leq 2p$ and $|K|\leq q-1$. The last inequality contradicts the minimality of $K(y,k)$. Hence, $y\in M^k(p,q)$ which shows that $Y_k=\bigcup\{M^k(p,q):p,q=1,2,..\}$. 

Suppose $y\in M^k(p,q)$. Then there exists a set $K\in\mathcal A^k(y)$ with $a(y,K)\leq p$ and $|K|\leq q$. Since $y\not\in Y_{2p,q-1}^k$, $|K|=q$. If there exists another $K'\in\mathcal A^k(y)$ with $a(y,K')\leq p$ and $|K'|=q$, then $K\cap K'\neq\varnothing$, $|K\cap K'|\leq q-1$ and, by Claim 4, $a(y,K\cap K')\leq a(y,K)+ a(y,K')\leq 2p$. This means that $y\in Y_{2p,q-1}^k$, a contradiction. Hence, there exists a unique $K_{kp}(y)\in\mathcal A^k(y)$ such that $a(y,K_{kp}(y))\leq p$ and $|K_{kp}(y)|=q$. 

For every $q$ let $[H_k]^q$ denote the set of all $q$-points subsets of $H_k$ endowed with the Vietoris topology.   
\begin{claim}
The map $\Phi_{kpq}:M^k(p,q)\to [H_k]^q$, $\Phi_{kpq}(y)=K_{kp}(y)$, is continuous.
\end{claim}
Because $K_{kp}(y)\subset H_k$ consists of $q$ points for all $y\in M^k(p,q)$, it suffices to show that if $K_{kp}(y)\cap U\neq\varnothing$ for some open $U\subset H$, then there is a neighborhood $V$ of $y$ in $\overline Y$ with $K_{kp}(z)\cap U\neq\varnothing$ for all $z\in V\cap M^k(p,q)$. We can assume that $K_{kp}(y)\cap U$ contains exactly one point $x_0$. 

Let $q\geq 2$, so $K_{kp}(y)=\{x_0,x_1,..,x_{q-1}\}$. Since $y\not\in Y_{2p,q-1}^k$ we have $a(y,K)>2p$, where $K=\{x_1,..,x_{q-1}\}$.
Hence, there are $\overline f, \overline g\in E(H)$ such that $|\overline f(x)-\overline g(x)|<1$ for all $x\in K$ and 
$|\alpha_y(\overline f)-\alpha_y(\overline g)|>2p$. The last inequality implies $\overline f(x_0)\neq \overline g(x_0)$, otherwise $a(y,K_{kp}(y))$ would be greater than $2p$ (recall that $y\in M^k(p,q)$ implies $a(y,K_{kp}(y))\leq p$). So, at least one of the numbers $\overline f(x_0), \overline g(x_0)$ is not zero. Without loss of generality we can assume that $\overline f(x_0)>0$, and let $r$ be a rational number with $\frac{-1+\delta}{\overline f(x_0)}<r<\frac{1+\delta}{\overline f(x_0)}$, where $\delta=\overline f(x_0)-\overline g(x_0)$. Then
$-1<(1-r)\overline f(x_0)-\overline g(x_0)<1$, and choose $\overline h_1\in E(H)$ such that $\overline h_1(x_0)=r$ and $\overline h_1(x)=0$ for all $x\not\in U$. Consider the function $\overline h=(1-\overline h_1)\overline f$. Clearly, $\overline h(x_0)=(1-r)\overline f(x_0)$ and $\overline h(x)=\overline f(x)$ if $x\not\in U$. Hence, 
$\overline h\in E(H)$ and $|\overline h(x)-\overline g(x)|<1$ for all $x\in K_{kp}(y)$. This implies $|\alpha_y(\overline h)-\alpha_y(\overline g)|\leq p$. Then
$$|\alpha_y(\overline f)-\alpha_y(\overline h)|\geq |\alpha_y(\overline f)-\alpha_y(\overline g)|-|\alpha_y(\overline h)-\alpha_y(\overline g)|>2p-p=p.$$
Observe that it is not possible $\alpha_y(\overline f)=\alpha_y(\overline h)=\pm\infty$ because $|\alpha_y(\overline h)-\alpha_y(\overline g)|\leq p$ would imply $\alpha_y(\overline g)=\pm\infty$. Then $|\alpha_y(\overline f)-\alpha_y(\overline g)|=0$, a contradiction.
 
The set $V=\{z\in\overline Y:|\alpha_z(\overline f)-\alpha_z(\overline h)|>p\}$ is a neighborhood of $y$. Since $\overline h(x)=\overline f(x)$ for all $x\not\in U$, 
$K_{kp}(z)\cap U=\varnothing$ for some $z\in V\cap M^k(p,q)$ would imply $|\alpha_z(\overline h)-\alpha_z(\overline f)|\leq p$, a contradiction. Therefore,  
$K_{kp}(z)\cap U\neq\varnothing$ for $z\in V\cap M^k(p,q)$.

If $q=1$, then $K_{kp}(y)=\{x_0\}$ and $K(y,k)=K_{kp}(y)$. So, $H_k\backslash U\not\in\mathcal A^k(y)$ (otherwise $K(y,k)\subset H_k\backslash U$). 
Hence, there exist $\overline f,\overline g\in E(H)$ such that $|\overline f(x)-\overline g(x)|<1$ for all $x\in H_k\backslash U$ and 
$|\alpha_y(\overline f)-\alpha_y(\overline g)|>2p$. Define $\overline h\in E(H)$ as in the previous case and use the same arguments to complete the proof.

Since $Y_{p,q}^k$ are compact subsets of $\overline Y$, each $M^k(p,q)$ is a countable union of compact subsets $\{F_n^k(p,q):n=1,2,..\}$ of $\overline Y$. So, by Claim 6,
$Y_k=\bigcup\{F_n^k(p,q):n,p,q=1,2,..\}$. According to Claim 7, all maps $\Phi_{kpq}^n=\Phi_{kpq}|F_n^k(p,q):F_n^k(p,q)\to [H_k]^q$ are continuous.
Moreover, since $Y\subset\bigcup_kY_k$, $Y\subset\bigcup\{F_n^k(p,q):n,p,q,k=1,2,..\}$. 
\begin{claim}
The fibers of $\Phi_{kpq}^n:F_n^k(p,q)\to [H_k]^q$ are finite. 
\end{claim} 
We follow the arguments from the proof of \cite[Theorem 4.2]{gkm}. Fix $z\in F_n^k(p,q)$ for some $n,p,q, k$ and let $A=\{y\in F_n^k(p,q):K_{kp}(y)=K_{kp}(z)\}$. Since $\Phi_{kpq}^n$ is a perfect map, $A$ is compact. Suppose $A$ is infinite, so it contains a convergent sequence $S=\{y_m\}$ of distinct points. Because $E(\overline Y)$ contains a $QS$-algebra $\Gamma$ on $\overline Y$, for every $y_m$ there exist its neighborhood $U_m$ in $\overline Y$ and a function $\overline g_m\in\Gamma$, $\overline g_m:\overline Y\to [0,2p]$ such that: $U_m\cap S=\{y_m\}$,
$\overline g_m(y_m)=2p$ and $\overline g_m(y)=0$ for all $y\not\in U_m$. Since $\varphi$ is $c$-good, for each $m$ there is $f_m\in E(X)$ with $\varphi (f_m)=\overline g_m|Y=g_m$ and $||f_m||\leq c||g_m||$. So, $||\overline f_m||\leq 2pc$, $m=1,2,..$ and the sequence $\{\overline f_m\}$ is contained in the compact set $[-2pc,2pc]^{H}$. Hence,  $\{\overline f_m\}$ has an accumulation point in  $[-2pc,2pc]^{H}$. This implies the existence of $i\neq j$ such that $|\overline f_i(x)-\overline f_j(x)|<1$ for all $x\in K_{kp}(z)$. Consequently, since $K_{kp}(y_j)=K_{kp}(z)$, $|\alpha_{y_j}(\overline f_j)-\alpha_{y_j}(\overline f_i)|\leq p$. On the other hand, $\alpha_{y_j}(\overline f_j)=\overline{\varphi(f_j)}(y_j)=\overline g_j(y_j)=2p$ and
$\alpha_{y_j}(\overline f_i)=\overline{\varphi(f_i)}(y_j)=\overline g_i(y_j)=0$, so $|\alpha_{y_j}(\overline f_j)-\alpha_{y_j}(\overline f_i)|=2p$, a contradiction.

Now, we can complete the proof of Proposition 2.1. Suppose $H$ has a property $\mathcal P$ satisfying conditions $(a)-(c)$. Then so does $H_k^q$ for each $k,q$ because $H_k^q$ is closed in $H^q$. We claim that the space $[H_k]^q$ also has the property $\mathcal P$. Indeed, let $\mathcal C$ be a countable base of $H_k$.  For every $q$-tuple $(U_1, . .,U_q)$ of elements of $\mathcal C$ with pairwise disjoint closures, the closed set
$$W(U_1, . . , U_q) =\{\{x_1, . . ,x_q\}:x_i\in\overline U_i, i=1,...,q\}\subset [H_k]^q$$ is homeomorphic to the closed subset 
$\overline U_1\times· · ·\times\overline U_q$ of $H_k^q$. Hence, $W(U_1, . . . ,U_q)\in\mathcal P$. Clearly, the space $[H_k]^q$ can be covered by countably many sets of the form $W(U_1, . . , U_q)$, therefore it belongs to $\mathcal P$. 
Finally, since the maps $\Phi_{kpq}^n:F_n^k(p,q)\to [H_k]^q$
are perfect and have finite fibers, each $F_n^k(p,q)$ has the property $\mathcal P$. 
Therefore, by condition $(b)$, $Y_\infty=\bigcup\{F_n^k(p,q):n,p,q,k=1,2,..\}$ has the property $\mathcal P$. 

It remains to show that all powers of $Y_\infty$ also have the property $\mathcal P$. Simplifying the notations, we observed that $Y_\infty=\bigcup_{m=1}^\infty F_m$, where every $F_m$ is a compact set admitting a map $\Phi_m$ with finite fibers onto a compact subset of $H^m$. 
Then for every $k$ we have $$Y_\infty^k=\bigcup_{(m_1,m_2,...,m_k)}F_{m_1}\times F_{m_2}\times...\times F_{m_k}.$$ Consequently, 
$\prod_{i=1}^k\Phi_{m_i}:\prod_{i=1}^kF_{m_i}\to H^{m_1+m_2+..+m_k}$ is a map which fibers are products of $k$-many finite sets. Hence, the fibers of
$\prod_{i=1}^k\Phi_{m_i}$ are finite. Because $H^{m_1+m_2+..+m_k}\in\mathcal P$, so is $\prod_{i=1}^kF_{m_i}$. Finally, by property $(b)$, 
$Y_\infty^k\in\mathcal P$.  
\end{proof}
All definitions below, except that one of $(m-C)$-spaces, can be found in \cite{en}.
 A normal space $X$ is called strongly countable-dimensional if $X$ can be represented as a countable union of closed finite-dimensional subspaces. Recall that a normal space $X$ is weakly infinite-dimensional if for every sequence $\{(A_i,B_i)\}$ of pairs of disjoint closed subsets of $X$ there exist closed sets $L_1,L_2,..$ such that $L_i$ is a partition between $A_i$ and $B_i$ and $\bigcap_iL_i=\varnothing$. A normal space $X$ is a $C$-space if for every sequence $\{\mathcal G_i\}$ of open covers of $X$ there exists a sequence $\{\mathcal H_i\}$ of families of pairwise disjoint open subsets of $X$ such that for 
 $i=1,2,..$ each member of $\mathcal H_i$ is contained in a member of $\mathcal G_i$ and the union $\bigcup_i\mathcal H_i$ is a cover of $X$. The $(m-C)$-spaces, where $m\geq 2$ is a natural number, were introduced by Fedorchuk \cite{vf}: A normal space $X$ is an $(m-C)$-space if for any sequence $\{\mathcal G_i\}$ of open covers of $X$ such that each $\mathcal G_i$ consists of at most $m$ elements, there is a sequence of disjoint open families $\{\mathcal H_i\}$ such that each $\mathcal H_i$ refines $\mathcal G_i$ and $\bigcup_i\mathcal H_i$ is a cover of $X$. The $(2-C)$-spaces are exactly the weakly infinite-dimensional spaces and for every $m$ we have the inclusion $(m+1)-C\subset m-C$. Moreover, every $C$-space is $m-C$ for all $m$.

 It is well known that the class of metrizable strongly countable-di\-men\-sional spaces contains all finite-dimensional metrizable spaces and is contained in the class of metrizable $C$-spaces. The last inclusion follows from the following two facts: (i) every finite-dimensional paracompact space is a $C$-space \cite[Theorem 6.3.7]{en};
(ii) every paracompact space which is a countable union of its closed $C$-spaces is also a $C$-space \cite[Theorem 4.1]{gv}.
 Moreover, every $C$-space is weakly infinite-dimensional \cite[Theorem 6.3.10]{en}.

In the class of $\sigma$-compact metrizable spaces the zero-dimensionality satisfies all conditions $(a)-(c)$, see 
Theorems 1.5.16, 1.2.2 and 1.12.4 from \cite{en}.  The strong countable-dimensionality also satisfies all these conditions, condition $(c)$ follows easily from 
\cite[Theorem 1.12.4]{en}. For $C$-space this follows from mentioned above fact that a countable union of closed compact $C$-spaces is a $C$-space \cite[Theorem 4.1]{gv} and the following results of Hattori-Yamada \cite{hy}: the class of compact $C$-spaces is closed under finite products any perfect preimage of a $C$-space with $C$-space fibers is a $C$-space. Finally, if a $\sigma$-compact space $X$ is weakly infinite-dimensional, then obviously every closed subset of $X$, as well as any countable union of closed subsets of $X$ have the same property, see \cite[Theorem 6.1.6]{en}. Condition $(c)$ follows from the following result of Pol \cite[Theorem 4.1]{po}: If $f:X\to Y$ is a continuous map between compact metrizable spaces such that $Y$ is 
weakly infinite-dimensional and each fiber $f^{-1}(y)$, $y\in Y$, is at most countable, then $X$ is weakly infinite-dimensional.  The validity of conditions $(a)-(c)$ for $(m-C)$-spaces in the class of $\sigma$-compact metrizable spaces follows from the following results \cite{vf} : The $(m-C)$-space property is hereditary with respect to closed subsets, a countable union of closed $(m-C)$-spaces is also $m-C$. Moreover, for $(m-C)$-spaces Krupski \cite[Lemma 4.5]{k1} established an analogue of the cited above Pol's result. Therefore, condition $(c)$ holds for the property $m-C$ in the class of $\sigma$-compact metrizable spaces.  

\section{Uniformly continuous surjections} 
In this section we prove Theorem 1.1 and Theorem 1.3.

For every space $X$ let $\mathcal F_X$ be the class of all maps from $X$ onto second countable spaces. For any two maps $h_1,h_2\in\mathcal F_X$ we write
$h_1\succ h_2$ if there exists a continuous map $\theta:h_1(X)\to h_2(X)$ with $h_2=\theta\circ h_1$. If $\Phi\subset C(X)$ we denote by $\triangle\Phi$ the diagonal product of all $f\in\Phi$. Clearly, $(\triangle\Phi)(X)$ is a subspace of the product $\prod\{\mathbb R_f:f\in\Phi\}$, and 
let $\pi_f:(\triangle\Phi)(X)\to\mathbb R_f$ be the projection.
Following \cite{gu}, we call a set $\Phi\subset C(X)$ {\em admissible} if the family
$\pi(\Phi)=\{\pi_f:f\in\Phi\}$ is a $QS$-algebra on $(\triangle\Phi)(X)$. 
We are using the following facts:
\begin{itemize}
\item[(3.1)] $\dim X\leq n$ if and only if for every $h\in\mathcal F_X$ there exists a $h_0\in\mathcal F_X$ such that $\dim h_0(X)\leq n$ and $h_0\succ h$ \cite{p}.
\item[(3.2)] If $\dim X\leq n$ and $\Phi\subset C(X)$ is countable, then there exists a countable set $\Theta\subset C(X)$ containing $\Phi$ with $\dim(\triangle\Theta)(X)\leq n$. Moreover, it follows from the proof of \cite[Lemma 2.2]{gu} that we can choose $\Theta\subset C^*(X)$ provided that
    $\Phi\subset C^*(X)$.
\item[(3.3)] For every countable $\Phi'\subset C(X)$ there is a countable admissible set $\Phi$ containing $\Phi'$ such that $(\triangle\Phi)(X)$ is homeomorphic to $(\triangle\Phi')(X)$. According to the proof of \cite[Lemma 2.4]{gu}, $\Phi$ could be taken to be a subset of $C^*(X)$ if $\Phi'\subset C^*(X)$. Moreover, we can assume that $\Phi'$ satisfies condition $(2.3)$, so $\pi(\Phi)$ also satisfies that condition.
\item[(3.4)] If $\{\Psi_n\}$ is an increasing sequence of admissible subsets of $C(X)$, then $\Psi=\bigcup_n\Psi_n$ is also admissible, see \cite[Lemma 2.5]{gu}.
    \end{itemize}

We also need the following lemmas:
\begin{lem}
Let $X$ be a $0$-dimensional separable metrizable space and $E(X)$ be a countable subfamily of $C^*(X)$. Then there exists a metrizable $0$-dimensional compactification $\overline X$ of $X$ such that each $f\in E(X)$ can be extended over $\overline X$.
\end{lem} 
\begin{proof}
Let $X_0$ be a metrizable compactification of $X$ and for every $f\in E(X)$ denote by $Z_f$ the closure of $f(X)$ in $\mathbb R$. Consider the diagonal product $h$ of the maps $j$ and $\triangle\{f:f\in E(X)\}$, where $j:X\hookrightarrow X_0$ is the embedding. Then the closure $X_1$ of $h(X)$ in the product
$X_0\times\prod_{f\in E(X)}Z_f$ is a compactification of $X$ such that every $f\in E(X)$ can be extended over $X_1$. 
Let $\theta:\beta X\to X_1$ be the map witnessing that $\beta X$ is a compactification
of $X$ larger than $X_1$. Since $\dim\beta X=0$, by the Marde\v{s}i\'{c} factorization theorem \cite[Theorem 3.4.1]{en} there is a metrizable compactum $\overline X$ and maps $\nu:\beta X\to\overline X$ and $\eta:\overline X\to X_1$ such that $\dim\overline X=0$ and $\theta=\eta\circ\nu$. Evidently, $\nu|i(X)$ is a homeomorphism, where $i:X\hookrightarrow\beta X$ is the embedding, so $\overline X$ is a compactification of $X$. Because every $f\in E(X)$ is extendable to a function $\overline f:X_1\to\mathbb R$, the composition $\overline f\circ\eta$ is an extension of $f$ over $\overline X$. 
\end{proof}

\begin{lem}
Let $X$ be a separable metrizable space and $E(X)$ be a countable subfamily of $C(X)$. Then there exists a metrizable compactification $\overline X$ of $X$ such that each $f\in E(X)$ can be extended to a map $\overline f:\overline X\to\overline{\mathbb R}$. Moreover, if $\dim X=0$, then we can assume that
$\dim\overline X=0$. 
\end{lem}  
\begin{proof}
The proof is similar to the proof of Lemma 3.1. The only difference is that for every $f\in E(X)$ we consider $Z_f$ to be the closure of $f(X)$ in $\overline{\mathbb R}$.  
\end{proof}  

\textit{Proof of Theorem $1.1$.} Let $T:D_p(X)\to D_p(Y)$ be a uniformly continuous $c$-good surjection. Everywhere below for $f\in C^*(X)$ let $\overline f\in C(\beta X)$ be its extension; similarly if $g\in C(Y)$ then $\overline g\in C(\beta Y,\overline{\mathbb R})$ is the extension of $g$.
According to $(3.1)$, it suffices to prove that for every
$h\in\mathcal F_{Y}$ there is $h_0\in\mathcal F_{Y}$ such that $\dim h_0(Y)=0$ and $h_0\succ h$.
So, we fix $h\in\mathcal F_{Y}$ and let $\overline h:\beta Y\to\overline{h(Y)}$ be an extension of $h$, where $\overline{h(Y)}$ is a metrizable compactification of $h(Y)$. We will construct by induction two sequences $\{\Psi_n\}_{n\geq 1}\subset C(\beta X)$ and $\{\Phi_n\}_{n\geq 1}\subset C(\beta Y,\overline{\mathbb R})$ of countable sets, countable $QS$-algebras $\mathcal A_n$ on $(\triangle\Psi_n)(\beta X)$ and  
countable $QS$-algebras $\Lambda_{n}$ on $Y_n'=(\triangle\Phi_n')(\beta Y)$, where
$\Phi_n'=\{\overline{T(f)}:\overline f\in\Psi_n\}$, 
satisfying the following conditions for every $n\geq 1$:
\begin{itemize}
\item[(3.5)] $\Phi_1\subset C(\beta Y)$ is admissible  and $\triangle\Phi_1\succ\overline h$;
\item[(3.6)] $\Phi_n\subset\Phi_{n+1}=\Phi_n'\cup\{\lambda\circ(\triangle\Phi_n'):\lambda\in\Lambda_{n}\}$;
\item[(3.7)] Each $\Psi_n$ is admissible, $\dim(\triangle\Psi_n)(\beta X)=0$ and $\Psi_n\subset\Psi_{n+1}$;
\item[(3.8)] $\mathcal A_n$ is a $QS$-algebra on $(\triangle\Psi_n)(\beta X)$ satisfying condition $(2.3)$; 
\item[(3.9)] $\Lambda_{n+1}$ contains $\{\lambda\circ\delta_{n}:\lambda\in\Lambda_{n}\}$,  where 
$\delta_{n}:Y'_{n+1}\to Y_{n}'$ is the surjective map generated by the inclusion $\Phi_{n}'\subset\Phi_{n+1}'$;
\item[(3.10)] For every $\overline{g}\in\Phi_n\cap C(\beta Y)$ there is $\overline f_g\in\Psi_{n}$ with $||f_g||\leq c||g||$ and $T(f_g)=g$. 
\end{itemize}

Since $\overline h(\beta Y)$ is a separable metrizable space, 
by $(3.3)$, there is a countable admissible set $\Phi_1\subset C(\beta Y)$ with $\triangle\Phi_1\succ\overline h$. 
Choose a countable set $\Psi_1'\subset C(\beta X)$ such that for every $\overline g\in\Phi_1$ there is $\overline f_g\in\Psi_1'$ with $||f_g||\leq c||g||$ and $T(f_g)=g$. Next, use $(3.2)$ to find  countable $\Theta_1\subset C(\beta X)$ containing $\Psi_1'$ such that $\dim(\triangle\Theta_1)(\beta X)=0$. Finally, by $(3.3)$, we can extend  $\Theta_1$ to a countable admissible set 
$\Psi_1\subset C(\beta X)$ such that $(\triangle\Psi_1)(\beta X)$ is homeomorphic to $(\triangle\Theta_1)(\beta X)$ and the $QS$-algebra $\mathcal A_1=\pi(\Psi_1)$ on $(\triangle\Psi_1)(\beta X)$ satisfies condition $(2.3)$. Evidently, $\dim(\triangle\Psi_1)(\beta X)=0$. 

Suppose $\Phi_k$ and $\Psi_k$ are already constructed for all $k\leq n$. Then $\Phi_{n}'=\{\overline{T(f)}:\overline f\in\Psi_n\}$ is a countable set in $C(\beta Y,\overline{\mathbb R})$. Because $\Psi_{n-1}\subset\Psi_n$, $\Phi_{n-1}'\subset\Phi_n'$. So, there is a surjective map $\delta_{n-1}: Y_n'\to\ Y_{n-1}'$, see $(3.9)$. Since $\triangle\Phi_n'\succ\triangle\Phi_{n-1}'$ we have $\delta_{n-1}(y)=\triangle\Phi_{n-1}'((\triangle\Phi_n')^{-1}(y))$ for all $y\in Y_n'$.
 Choose a countable $QS$-algebra $\Lambda_{n}$ on $Y_n'$ containing the family $\{\lambda\circ\delta_{n-1}:\lambda\in\Lambda_{n-1}\}$ 
 and let $\Phi_{n+1}=\Phi_n'\cup\{\lambda\circ(\triangle\Phi_n'):\lambda\in\Lambda_{n}\}$.
Next, take a countable set 
 $\Psi_{n+1}'\subset C(\beta X)$ containing $\Psi_n$ such that for every $\overline{g}\in\Phi_{n+1}\cap C(\beta Y)$ there is $\overline f_g\in\Psi_{n+1}'$ with $||f_g||\leq c||g||$ and $T(f_g)=g$.
Then, by $(3.2)$ there is countable $\Theta_{n+1}\subset C(\beta X)$ containing  $\Psi_{n+1}'$ with $\dim(\triangle\Theta_{n+1})(\beta X)=0$. Finally, according to $(3.3)$, we extend $\Theta_{n+1}$ to a countable admissible set $\Psi_{n+1}\subset C(\beta X)$ such that $(\triangle\Psi_{n+1})(\beta X)$ is homeomorphic to $(\triangle\Theta_{n+1})(\beta X)$ and the $QS$-algebra $\mathcal A_{n+1}=\pi(\Psi_{n+1})$ on $(\triangle\Psi_{n+1})(\beta X)$ satisfies condition $(2.3)$. This completes the induction.

Let $\Psi=\bigcup_n\Psi_n$, $X_0=(\triangle\Psi)(X)$, $\overline X_n=(\triangle\Psi_n)(\beta X)$ and $\overline X_0=(\triangle\Psi)(\beta X)$. Similarly, let $\Phi=\bigcup_n\Phi_n$, $Y_0=h_0(Y)$ and 
$\overline Y_0=(\triangle\Phi)(\beta Y)$, where $h_0=(\triangle\Phi)|Y$. 
Both $\Psi$ and $\Phi$ are countable and $\Psi$ is an admissible subset of $C(\beta X)$, see $(3.4)$. Hence, 
the family $E(\overline X_0)=\{\pi_{\overline f}:\overline f\in\Psi\}$ is a countable $QS$-algebra on $\overline X_0$. 
Moreover, the family $E(Y_0)=\{\pi_{\overline g}|Y_0:\overline g\in\Phi\}$
is extendable over $\overline Y_0$.
Since $\Psi_n\subset\Psi_{n+1}$ for every $n$, there are maps $\theta_n^{n+1}:\overline X_{n+1}\to \overline X_n$. Because $\Psi=\bigcup_n\Psi_n$, the space $\overline X_0$ is the limit of the inverse sequence $S_X=\{\overline X_n,\theta_n^{n+1}\}$ with $\dim\overline X_n=0$ for all $n$. Hence, $\overline X_0$ is also $0$-dimensional, see \cite[Theorem 3.4.11]{en}. Observe also that $E(\overline X_0)=\bigcup_{n=1}^\infty\{h\circ\theta_n:h\in\mathcal A_n\}$.

Let us show that $E(\overline X_0)$ ia a $QS$-algebra satisfying condition $(2.3)$. Because $E(\overline X_0)$ is the union of the increasing sequence of the families $\{h\circ\theta_n:h\in\mathcal A_n\}$ and each $\mathcal A_n$ is a $QS$-algebra, $E(\overline X_0)$ is closed under sums, multiplications and multiplications by rational numbers. It remains to show that for every $x\in\overline X_0$ and every its neighborhood $U\subset\overline X_0$ there is 
$\overline f\in E(\overline X_0)$ such that $\overline f(x)=1$ and $\overline f(\overline X_0\backslash U)=0$. But that follows from the proof of the more general condition $(2.3)$.
To show that $E(\overline X_0)$ satisfies $(2.3)$,
take a compact set $K\subset\overline X_0$ and an open set $W\subset\overline X_0$ containing $K$. Because $\overline X_0$ is the limit of the inverse sequence $S_X=\{\overline X_n,\theta_n^{n+1}\}$,
there are $n$, a compact set $K_n\subset\overline X_n$ and an open set $W_n\subset\overline X_n$ containing $K_n$ such that 
$K\subset\theta_n^{-1}(K_n)\subset\theta_n^{-1}(W_n)\subset W$, where $\theta_n:\overline X_0\to\overline X_n$ denotes the $n$-th projection in $S_X$. Then, according to condition $(3.8)$, there is $f_n\in\mathcal A_n$ with $f_n|K_n=1$, $f_n|(\overline X_n\backslash W_n)=0$ and $f_n:\overline X_n\to [0,1]$.
Finally, observe that the function $\overline f=f_n\circ\theta_n:\overline X_0\to [0,1]$ belongs to $E(\overline X_0)$ and $\overline f$ satisfies the conditions $\overline f|K=1$ and 
$\overline f|(\overline X_0\backslash W)=0$.

It follows from the construction that $\Phi=\{\overline{T(f)}:\overline f\in\Psi\}$ and for every $\overline g\in\Phi\cap C(\beta Y)$ there is $\overline f_g\in\Psi$ with $T(f_g)=g$ and 
$||f_g||\leq c||g||$. Observe that for all $\overline f\in\Psi$ and $\overline g\in\Phi$ we have $f=\pi_{\overline f}\circ(\triangle\Psi)|X$ and $g=\pi_{\overline g}\circ (\triangle\Phi)|Y$. Therefore, there is a surjective map $\varphi:E_p(X_0)\to E_p(Y_0)$ defined by $\varphi(\pi_f)=\pi_{T(f)}$,
where $\pi_f$ and $\pi_g$ denote, respectively, the functions $\pi_{\overline f}|X_0$ and $\pi_{\overline g}|Y_0$. 
 Moreover $||f||=||\pi_f||$ and $||g||=||\pi_g||$ for all $\overline f\in\Psi$ and $\overline g\in\Phi\cap C(\beta Y)$. 
This implies that $\varphi$ is a $c$-good surjection. 

Let's show that $\varphi$ is uniformly continuous. Suppose
$$V=\{\pi_g:|\pi_g(y_i)|<\varepsilon{~}\forall i\leq k\}$$ is a neighborhood of the zero function in $E_p(Y_0)$. Take points $\overline y_i\in Y$ with $h_0(\overline y_i)=y_i$, $i=1,2,..,k$, and let 
$\widetilde V=\{g\in D_p(Y):|g(\overline y_i)|<\varepsilon{~}\forall i\leq k\}$. Since $T$ is uniformly continuous, there is a neighborhood 
$$\widetilde U=\{f\in D_p(X):|f(\overline x_j)|<\delta{~}\forall j\leq m\}$$ of the zero function in $D_p(X)$ such that $f-f'\in\widetilde U$ implies $T(f)-T(f')\in \widetilde V$ for all $f,f'\in D_p(X)$. Let $x_j=(\triangle\Psi)(\overline x_j)$ and 
$$U=\{\pi_f:|\pi_f(x_j)|<\delta{~}\forall j\leq m\}.$$ Obviously, $\pi_f-\pi_{f'}\in U$ implies $f-f'\in\widetilde U$. Hence, $T(f)-T(f')\in\widetilde V$, which yields $\varphi(\pi_f)-\varphi(\pi_{f'})\in V$. 

Finally, we can show that $E(\overline Y_0)$ contains a $QS$-algebra on $\overline Y_0$. Since $\Lambda_n\subset C(Y_n')$ is a $QS$-algebra on $Y_n'$, it separates the points and the closed sets in $Y_n'$. So, $Y_n'$ is homeomorphic to $(\triangle\Lambda_n)(Y_n')$. This implies that $Y_{n+1}=(\triangle\Phi_{n+1})(\beta Y)$ is homeomorphic to $Y_n'$. Therefore, $\Lambda_n$ can be considered as a $QS$-algebra on $Y_{n+1}$. 
On the other hand $\overline Y_0$ is the limit of the inverse sequence $S_Y=\{Y_{n+1},\gamma_{n+1}^{n+2},n\geq 1\}$, where $\gamma_{n+1}^{n+2}:Y_{n+2}\to Y_{n+1}$ is the surjective map generated by the inclusion $\Phi_{n+1}\subset\Phi_{n+2}$. According to condition $(3.9)$, we can also assume that $\Lambda_{n+1}$ contains the family $\{\lambda\circ\gamma_{n+1}^{n+2}:\lambda\in\Lambda_n\}$. 
Denote by $\gamma_{n+1}:\overline Y_0\to Y_{n+1}$ the projections 
in $S_Y$, and let $\Gamma_n=\{\lambda\circ\gamma_{n+1}:\lambda\in\Lambda_n\}$. Then $\{\Gamma_n\}_{n\geq 1}$ is an increasing sequence of countable families and $\Gamma_n\subset E(\overline Y_0)=\{\pi_{\overline g}:\overline g\in\Phi\}$ for every $n$. We claim that $\Gamma=\bigcup_n\Gamma_n$ is a $QS$-algebra on $\overline Y_0$. Indeed, $\Gamma$ is closed under addition, multiplication and multiplication by rational numbers. Because for every $y\in\overline Y_0$ and its neighborhood $V\subset\overline Y_0$ there is $n$ and an open set $V_n\subset Y_{n+1}$ such that $y_n=\gamma_{n+1}(y)\in V_n$ and 
 $\gamma_{n+1}^{-1}(V_n)\subset V$, there exists $\lambda\in\Lambda_n$ with $\lambda(y_n)=1$ and $\lambda(Y_{n+1}\backslash V_n)=0$. Then $g=\lambda\circ\gamma_{n+1}\in\Gamma$, $g(y)=1$ and $g(\overline Y_0\backslash V)=0$.   
 
To prove Theorem 1.1,  we apply Proposition 2.1 with $H=\overline X_0$ to find a $\sigma$-compact set $Y_\infty\subset\overline Y_0$ containing $Y_0$ with $\dim Y_\infty=0$.
Therefore, by \cite[Proposition 1.2.2]{en}, $\dim Y_0=0$. $\Box$

\begin{pro}\cite{l}
For every linear continuous surjective map $T:C_p^*(X)\to C_p^*(Y)$ there is $c>0$ such that $T$ is $c$-good.
\end{pro}
\begin{proof}
By the Closed Graph Theorem, $T$ considered as a map between the Banach spaces $C^*_u(X)$ and $C^*_u(Y)$, both equipped with the sup-norm, is continuous. 
Then $T$ induced a linear isomorphism $T_0$ between $C^*_u(X)/K$ and $C^*_u(Y)$, where $K$ is the kernel of $T$. So, for every $g\in C_u^*(Y)$ we have 
$||T_0^{-1}(g)||\leq ||T_0^{-1}||\cdot||g||$. Because  $$||T_0^{-1}(g)||=\inf\{||f-h||:h\in K\},$$ where $f\in C_u^*(X)$ with $T(f)=g$, there exists $h_g\in K$ such that 
$||f-h_g||\leq 2||T_0^{-1}(g)||$. Hence, $||f-h_g||\leq 2||T_0^{-1}||\cdot||g||$. Since $T(f-h_g)=T(f)=g$, we obtain that $T$ is $c$-good with $c=2||T_0^{-1}||$.
\end{proof}

\textit{Proof of Theorem $1.3$.} Following the proof of Theorem $1.1$, we construct two sequences $\{\Psi_n\}_{n\geq 1}\subset C(\beta X)$ and $\{\Phi_n\}_{n\geq 1}\subset C(\beta Y,\overline{\mathbb R})$ of countable sets and countable $QS$-algebras $\mathcal A_n$ on $\triangle\Psi_n(\beta X)$ and  $\Lambda_{n}$ on $Y_n'=(\triangle\Phi_n')(\beta Y)$ satisfying the conditions $(3.5)-(3.10)$ except $(3.7)$. Because $X$ and $Y$ are separable metrizable spaces, we can choose countable sets $\Psi_1$ and $\Phi_1$ such that  
$(\triangle\Phi_1)|Y$ and $(\triangle\Psi_1)|X$ are homeomorphisms.

Then, following the notations from the proof of Theorem 1.1, we have that $X_0$ and $Y_0$ are homeomorphic to $X$ and $Y$, respectively. Moreover,
there exists a uniformly continuous $c$-good surjection $\varphi: E_p(X_0)\to E_p(Y_0)$ such that $E(\overline X_0)$ is a $QS$-algebra on $\overline X_0$
satisfying condition $(2.3)$, 
$E(\overline Y_0)\subset C(\overline Y_0,\overline{\mathbb R})$ and $E(\overline Y_0)$ contains a countable $QS$-algebra $\Gamma$ on $\overline Y_0$.
According to the proof of Proposition 2.1 with $H=X_0$, there is a $\sigma$-compact set $Y_\infty\subset\overline Y_0$ containing $Y_0$ which a countable union of closed sets $F\subset Y_\infty$ such that $F$ admits a map with finite fibres into a finite power of $X_0$. 
Let $Y_0=\bigcup_m Y_m$  with each $Y_m$ being compact. Then $Y_0$ can also be represented as a countable union of compact sets each admitting a map with finite powers into a finite power of $X_0$.
Now we apply the following fact which was actually used in the proof of Proposition 2.1: Assume all powers of a $\sigma$-compact metrizable space $P$ have a property $\mathcal P$ satisfying conditions $(a)-(c)$. If a metrizable space $Z$ is the union of countably many compact sets
$Z_n$ such that each $Z_n$ admits a map with finite fibres into a finite power $P^{k_n}$, then all finite powers of $Z$ also have the property $\mathcal P$. 
Since all finite powers of $X_0$ have the property $\mathcal P$, where $\mathcal P$ is either weakly infinite-dimensionality or the $(m-C)$-space property,
by mentioned above fact, all finite powers of $Y_0$ are either weakly infinite-dimensional or have the $(m-C)$-space property. 
$\Box$

\section{Proof of Theorem 1.4}
We consider topological properties $\mathcal P$ of separable metrizable spaces satisfying conditions $(b)$ from the introduction section plus the following three:
\begin{itemize}
\item[$(a')$] If $X\in\mathcal P$, then $F\in\mathcal P$ for every subset $F\subset X$;
\item[$(c')$] If $f:X\to Y$ is a perfect map between metrizable spaces with $0$-dimensional fibers and $Y\in\mathcal P$, then $X\in\mathcal P$;
\item[$(d')$] $\mathcal P$ is closed under finite products.
\end{itemize}
The $0$-dimensionality, the countable-dimensionality and the strong countable dimensionality satisfy these conditions, see \cite{en}.

If $E(X)\subset C(X)$ is a $QS$-algebra on $X$, then the family $LE(X)$ of all finite linear combinations $\sum_{i=1}^k\lambda_i\cdot f_i$ with $f_i\in E(X)$ and $\lambda_i\in\mathbb R$ is called the {\em linear hull} of $E(X)$.
\begin{pro}
Let $X$ and $Y$ be separable metrizable spaces and $E(X)\subset C(X)$ be a countable $QS$-algebra on $X$ and $E(Y)\subset C(Y)$ be a countable family. Suppose there are metrizable compactifications $\overline X$ and $\overline Y$ of $X$ and $Y$ and a countable base $\mathcal B$ of $\overline X$ such that: 
\begin{itemize}
\item Every $f\in E(X)$ can be extended to a map $\overline f:\overline X\to\overline{\mathbb R}$ and for every finite open cover $\gamma=\{U_i:i=1,2,..,k\}$ of $\overline X$ with elements from $\mathcal B$ there exists a partition of unity $\{\overline f_i:i=1,2,.,k\}$ subordinated to $\gamma$ with $f_i\in E(X)$;
\item Every $g\in E(Y)$ can be extended to a map $\overline g:\overline Y\to\overline{\mathbb R}$ and the set of all real-valued elements of $E_p(\overline Y)=\{\overline g:g\in E(Y)\}$ is dense in $C_p(\overline Y)$;
\item For every compact set $K\subset\overline X$ and every open set $W\subset\overline X$ containing $K$ there is $f\in E(X)$ such that $\overline f(K)=1$, $\overline f(\overline X\backslash W)=0$ and $f(x)\in [0,1]$ for all $x\in X$.    
\end{itemize}    
If $\overline X$ has a property $\mathcal P$ satisfying conditions $(a'),(b)$, $(c')$ and $(d')$, and
$\varphi: E_p(X)\to E_p(Y)$ is a linear continuous surjection such that $\varphi$ can be continuously extended over $LE_p(X)$, then $Y\in\mathcal P$.
\end{pro}
\begin{proof}

Let $E(\overline X)=\{\overline f:f\in E(X)\}$. 
Every $y\in\overline Y$ generates a map $l_y:E(\overline X)\to\overline{\mathbb R}$, $l_y(\overline f)=\overline{\varphi(f)}(y)$. 
Assuming the equalities from $(2.4)$, for any $\overline f_1,\overline f_2\in E(\overline X)$ and $x\in\overline X$, we can write $\overline f_1(x)+\overline f_2(x)$ but not always $\overline f_1+\overline f_2=\overline{f_1+f_2}$. Also, if $\overline f_1+\overline f_2\in E(\overline X)$, it is possible that
$l_y(\overline f_1+\overline f_2)\neq l_y(\overline f_1)+l_y(\overline f_2)$. 
If $\lambda\in\mathbb R\backslash\{0\}$ and $\overline{\lambda\cdot f}\in E(\overline X)$, then 
$l_y(\overline{\lambda\cdot f})=\lambda\cdot l_y(\overline f)$ (here, $\lambda\cdot (\pm\infty)=\pm\infty$ if $\lambda>0$ and 
$\lambda\cdot (\pm\infty)=\mp\infty$ if $\lambda<0$). In case $\lambda=0$, we have $l_y(\overline{0\cdot f})=0$. More general, if $\overline h\in C(\overline X)$ such that $h\cdot f\in E(X)$ for some $f\in E(X)$, then $l_y(\overline{h\cdot f})$ is well defined.

The {\em support of 
$l_y$}, $y\in\overline Y$, is defined to be the set $supp(l_y)$ of all $x\in\overline X$ satisfying the following condition \cite{vv}: for every neighborhood $U\subset\overline X$ of $x$ 
there is $f\in E(X)$ such that $\overline f(\overline X\backslash U)=0$ and $l_y(\overline f)\neq 0$. Obviously, $supp(l_y)$ is closed in $\overline X$. 
\begin{claim}
$supp(l_y)$ is non-empty for all $y\in\overline Y$.
\end{claim}
Indeed, because the set of real-valued functions from $E(\overline Y)$ is dense in $C_p(\overline Y)$, for every $y\in\overline Y$ there is $g\in E(\overline Y)$ with $g(y)\neq 0$. Then take $f_y\in E(X)$ such that $\varphi(f_y)=g|Y$, so $l_y(\overline f_y)=g(y)$. If $supp(l_y)=\varnothing$, every $x\in\overline X$ has a neighborhood  $V_x$ such that $l_y(\overline f)=0$ for any $f\in E(X)$ with $\overline f(\overline X\backslash V_x)=0$. Passing to smaller neighborhoods, we can assume that $V_x\in\mathcal B$ for all $x$. Hence, there is a finite open cover $\gamma=\{V_{x_1},..,V_{x_k}\}$ of $\overline X$ and a partition of unity $\mu=\{\overline h_1,..,\overline h_k\}\subset E(\overline X)$ subordinated to $\gamma$. Then $f_y\cdot h_i\in E(X)$ and $\overline{f_y\cdot h_i}|(\overline X\backslash V_{x_i})=0$ which implies $l_y(\overline{f_y\cdot h_i})=0$ for all $i$. Because $\mu$ is a partition of unity, 
$\overline f_y=\sum_{i=1}^k\overline{f_y\cdot h_i}$. So, $l_y(\overline f_y)=\sum_{i=1}^kl_y(\overline{f_y\cdot h_i})=0$, a contradiction. 

\begin{claim}
Let $U\subset\overline X$ be a neighborhood of $supp(l_y)$ and $f,g\in E(X)$ with $\overline f|U=\overline g|U$. Then $l_y(\overline f)=l_y(\overline g)$.
In particular, $\overline f(U)=0$ implies $l_y(\overline f)=0$.
\end{claim}
First, let $\overline f(U)=0$ for some $\overline f\in E(\overline X)$.
We can assume that $U$ is a finite union of open sets $V_i$ from $\mathcal B$, $i=1,..,k$. 
Every $x\in\overline X\backslash U$ has a neighborhood $V_x$ such that $l_y(\overline g)=0$ for any $g\in E(X)$ with $\overline g(\overline X\backslash V_x)=0$. As in the proof of Claim 9, we can assume that $V_x\in\mathcal B$, and choose another neighborhood $U_x\in\mathcal B$ with $\overline U_x\subset V_x$. Take a finite open cover $\{U_{x_1},..,U_{x_m}\}$ of $\overline X\backslash U$. Then 
$\gamma=\{V_1,..,V_k,U_{x_1},..,U_{x_m}\}$ is an open cover of $\overline X$ consisting of elements from $\mathcal B$. Hence, there exits a partition of unity $\{\overline h_1,.,\overline h_k,\overline\theta _1,..,\overline\theta_m\}\subset E(\overline X)$ subordinated to $\gamma$. Then $h_i\cdot f, \theta_j\cdot f\in E(X)$ for all $i,j$ and
$f=\sum_{i=1}^kh_i\cdot f+\sum_{j=1}^m\theta_j\cdot f$. Take a sequence $\{y_n\}\subset Y$ with $\lim_n y_n=y$. 
So, $\varphi(f)(y_n)=\sum_{i=1}^k\varphi(h_i\cdot f)(y_n)+\sum_{j=1}^m\varphi(\theta_j\cdot f)(y_n)$.
Observe that $(h_i\cdot f)(x)=0$ for all $x\in X$, so $\varphi(h_i\cdot f)$ is the zero function on $Y$ and $\varphi(h_i\cdot f)(y_n)=0$, $i=1,..,k$. 
Hence, $\varphi(f)(y_n)=\sum_{j=1}^m\varphi(\theta_j\cdot f)(y_n)$.
On the other hand, $(\theta_j\cdot f)(x)=0$ for all $x\in X\backslash U_{x_j}$. So, $(\overline{\theta_j\cdot f})(x)=0$ for all $x\in\overline X\backslash V_{x_j}$ 
because $\overline X\backslash V_{x_j}\subset\overline{X\backslash U_{x_j}}$. This implies that $l_y(\overline{\theta_j\cdot f})=0$ for all $j=1,..,m$. 
Since, $l_y(\overline{\theta_j\cdot f})=\lim_n\varphi(\theta_j\cdot f)(y_n)=0$, we have
$$l_y(\overline f)=\lim_n\varphi(f)(y_n)=\lim_n\sum_{j=1}^m\varphi(\theta_j\cdot f)(y_n)=0.$$
Suppose now that $\overline f|U=\overline g|U$ for some $f,g\in E(X)$. Then $f(x)-g(x)=0$ for all $x\in U\cap X$. Consequently, $(\overline{f-g})(x)=0$ for all $x\in U$ and, according to the previous paragraph, $l_y(\overline{f-g})=0$. Hence, for every sequence $\{y_n\}\subset Y$ converging to $y$ we have 
$$l_y(\overline{f-g})=\lim_n\varphi(f-g)(y_n)=\lim_n\varphi(f)(y_n)-\lim_n\varphi(g)(y_n)=l_y(\overline f)-l_y(\overline g).$$

\begin{claim}
If $supp(l_{y_0})\cap U\neq\varnothing$ for some open $U\subset\overline X$ and $y_0\in\overline Y$, then $y_0$ has a neighborhood $V\subset\overline Y$ such that $supp(l_{y})\cap U\neq\varnothing$ for every $y\in V$. 
\end{claim}
Let $x_0\in supp(l_{y_0})\cap U$ and $\overline f\in E(\overline X)$ be such that $\overline f(\overline X\backslash W)=0$ and $l_{y_0}(\overline f)\neq 0$, where $W$ is a neighborhood of $x_0$ with $\overline W\subset U$.
Assuming the claim is not true we can find a sequence $\{y_n\}\subset\overline Y$ converging to $y_0$ such that $supp(l_{y_n})\cap U=\varnothing$ for every $n$.
Since $\overline X\backslash\overline W$ is a neighborhood of each $supp(l_{y_n})$, by Claim 10,  $l_{y_n}(\overline f)=0$. Because 
$\lim_n l_{y_n}(\overline f)=l_{y_0}(\overline f)$, we have $l_{y_0}(\overline f)=0$, a contradiction.

Following the notations from the proof of Proposition 2.1, for every $y\in\overline Y$ we define
$$a(y)=\sup\{|l_y(\overline f)|:\overline f\in E(\overline X){~}\hbox{and}{~}|\overline f(x)|<1{~} \forall x\in supp(l_y)\}.$$
\begin{claim}
If $y\in Y$ then  $supp(l_{y})=\{x_1(y),x_2(y),..,x_q(y)\}$ is a finite non-empty subset of $X$. Moreover, there exist real numbers $\lambda_i(y)$, $i=1,..,q$, such that $\sum_{i=1}^q|\lambda_i(y)|=a(y)$ and
 $l_y(\overline f)=\sum_{i=1}^q\lambda_i(y)\overline f(x_i(y))$ for all $\overline f\in E(\overline X)$.
\end{claim}
Let $\widetilde\varphi:LE_p(X)\to LE_p(Y)$ be the continuous extension of $\varphi$. Then every $l_y$, $y\in Y$, can be continuously extended to
a linear functional $\widetilde l_y:LE_p(X)\to\mathbb R$, $\widetilde l_y(\sum_{i=1}^k\lambda_i\cdot f_i)=\sum_{i=1}^k\lambda_i\cdot l_y(f_i)$.
Since $\widetilde\varphi$ is uniformly continuous, according to the proof of Claim 1 from Proposition 2.1, for every $y\in Y$ there exists a finite set $K=\{x_1(y),x_2(y),..,x_q(y)\}\subset X$ such that 
$$\sup\{|\widetilde l_y(f)|:f\in LE(X){~}\hbox{and}{~}|f(x)|<1{~} \forall x\in K\}<\infty.$$ 
Because $\widetilde l_y$ is linear, we can show that $\widetilde l_y(g)=0$ for any $g\in LE(X)$ with $g(x_i(y))=0$, $i=1,..,q$. Since $E(X)$ is a $QS$-algebra, for every $i$ there is a function $g_i\in E(X)$ such that $g_i(x_i(y))=1$ and $g_i(x_j(y))=0$ with $j\neq i$. Now, for every $f\in E(X)$ the function
$g=f-\sum_{i=1}^qg_i\cdot f(x_i(y))$ belongs to $LE(X)$ and $g(x_i(y))=0$ for all $i$. So, $\widetilde l_y(g)=0$ and 
$l_y(f)=\sum_{i=1}^q\lambda_i(y)f(x_i(y))$, where with $\lambda_i(y)=l_y(g_i)$. Because $y\in Y$, we can also write $$l_y(\overline f)=\sum_{i=1}^q\lambda_i(y)\overline f(x_i(y))$$ with $\lambda_i(y)=l_y(\overline g_i)$.
Note that each $\lambda_i(y)$ is a real number because $l_y(g_i)=\varphi(g_i)(y)$. The last equality is valid for all $\overline f\in E(\overline X)$ and shows that $supp(l_y)=\{x_i(y):\lambda_i(y)\neq 0\}\subset K$. Note that, by Claim 9, $supp(l_y)\neq\varnothing$. 

To complete the proof of Claim 12, assume that $supp(l_y)=K$, and for every natural $k$ take a function $f_k\in E(X)$ with $f_k(x_i(y))=\varepsilon_i(1-1/k)$, where 
$\varepsilon_i=1$ if $\lambda_i(y)>0$ and $\varepsilon_i=-1$ if $\lambda_i(y)<0$. Clearly, $|f_k(x_i(y))|<1$ for all $i,k$ and 
$\lim_kl_y(\overline f_k)=\sum_{i=1}^q|\lambda_i(y)|$. Hence $\sum_{i=1}^q|\lambda_i(y)|\leq a(y)$. The reverse inequality $a(y)\leq \sum_{i=1}^q|\lambda_i(y)|$ follows from $l_y(\overline f)=\sum_{i=1}^q\lambda_i(y)\overline f(x_i(y))$, $\overline f\in E(\overline X)$.

For every $p,q\in\mathbb N$ let $Y_{p,q}=\{y\in\overline Y:|supp(l_y)|\leq q{~}\mbox{and}{~}a(y)\leq p\}$.
\begin{claim}
Every set $Y_{p,q}$ is closed in $\overline Y$.
\end{claim}
Let $\{y_n\}$ be a sequence in $Y_{p,q}$ converging to $y\in\overline Y$. Suppose $y\not\in Y_{p,q}$. Then either $supp(l_y)$ contains at least $q+1$ points or $a(y)>p$.
If $supp(l_y)$ contains at least $q+1$ points $x_1,x_2,..,x_{q+1}$, we choose disjoint neighborhoods $O_i$ of $x_i$, $i=1,..,q+1$. By Claim 11, there is a neighborhood $V$ of $y$ such that $supp(l_z)\cap O_i\neq\varnothing$ for all $i$ and $z\in V$. This implies that $supp(l_{y_n})$ has at least $q+1$ points for infinitely many $n$'s, a contradiction. 

If $y\not\in Y_{p,q}$ and $a(y)>p$, then there exists $\overline f\in E(\overline X)$ such that $|\overline f(x)|<1$ for all $x\in supp(l_y)$ and $|l_y(\overline f)|>p$. Take
a neighborhood $U$ of $supp(l_y)$ with $U\subset\{x\in\overline X:|\overline f(x)|<1\}$.
Choose another neighborhood $W$ of $supp(l_y)$ with $\overline W\subset U$. Next, there is $h\in E(X)$ such that $\overline h(\overline W)=1$, 
$\overline h(\overline X\backslash U)=0$ and $h(x)\in [0,1]$ for all $x\in X$. Then  
$\overline g=\overline{h\cdot f}\in E(\overline X)$ and $|\overline g(x)|<1$ for all $x\in\overline X$. Moreover,
$\overline g|W=\overline f|W$. So, by Claim 10, $|l_y(\overline g)|=|l_y(\overline f)|>p$. Therefore, $V=\{z\in\overline Y:|l_z(\overline g)|>p\}$ is a neighborhood of $y$ in $\overline Y$ with $V\cap Y_{p,q}=\varnothing$, a contradiction.

According to Claim 12, for every $y\in Y$ then there exist $p,q\in\mathbb N$ and real numbers $\lambda_i(y)$ such that for all $\overline f\in E(\overline X)$ we have $l_y(\overline f)=\sum_{i=1}^q\lambda_i(y)\overline f(x_i(y))$ with $a(y)=\sum_{i=1}^q|\lambda_i(y)|\leq p$, where $\{x_1(y),..,x_q(y)\}$ is the support $supp(l_y)$. 
Hence, $Y\subset\bigcup\{Y_{p,q}:p,q\in\mathbb N\}$. For every $p\geq 1$ and $q\geq 2$ we
define 
$$M(p,1)=Y_{p,1}{~}\hbox{and}{~}M(p,q)=Y_{p,q}\backslash Y_{2p,q-1}.$$ 
Some of the sets $M(p,q)$ could be empty, but $Y\subset\bigcup\{M(p,q):p,q=1,2,..\}$. Indeed, by Claim 12, there exist $p,q$ with $|supp(l_y)|=q$ and $a(y)\leq p$. Then $y\in M(p,q)$.
Since each $Y_{p,q}$ is closed, $M(p,q)=\bigcup_{n=1}^\infty F_n'(p,q)$ such that each $F_n'(p,q)$ is a compact subset of $\overline Y$. We define
$F_n(p,q)=\overline{Y\cap F_n'(p,q)}$. Then $Y\subset\bigcup\{F_n(p,q):n,p,q=1,2,..\}$. 
Obviously,
$supp(l_y)$ consists of $q$ different points for any $y\in M(p,q)$. So, we have a map $S_{p,q}:M(p,q)\to [\overline X]^q$, $S_{p,q}(y)=supp(l_y)$. According to Claim 11, $S_{p,q}$ is continuous when $[\overline X]^q$ is equipped with the Vietoris topology. 
For every $y\in M(p,q)$ let $S_{p,q}(y)=\{x_i(y)\}_{i=1}^q$. Everywhere below we consider the restriction $S_{p,q}|F_n(p,q)$ and for every $z\in F_n(p,q)$ denote by $A(z)=\{y\in F_n(p,q):S_{p,q}(y)=S_{p,q}(z)\}$ the fiber $S_{p,q}^{-1}(S_{p,q}(z))$ generated by $z$. Since $F_n(p,q)$ is compact and $S_{p,q}$ is continuous, each $A(z)$ is a compact subset of $F_n(p,q)$.
\begin{claim}
Let $z\in Y\cap F_n(p,q)$. Then for every $y\in A(z)$ there are real numbers  
$\{\lambda_i(y)\}_{i=1}^{q}$ such that $l_y(\overline f)=\sum_{i=1}^{q}\lambda_i(y)\overline f(x_i(z))$, where $S_{p,q}(z)=\{x_i(z)\}_{i=1}^{q}$, and $\sum_{i=1}^q|\lambda_i(y)|\leq p$ for 
all $\overline f\in E(\overline X)$. Moreover, each $\lambda_i$ is a continuous real-valued function on $A(z)$.
\end{claim}
Choose neighborhoods $O_i$ of $x_i(z)$ in $\overline X$ with disjoint closures and functions $\overline g_i\in E(\overline X)$ with $\overline g_i|O_i=1$ and $\overline g_i|O_j=0$ if $j\neq i$ (this can be done by choosing $g_i\in E(X)$ with $g_i(O_i\cap X)=1$ and $g_i(O_j\cap X)=0$ when $j\neq i$). According to the proof of Claim 12,
for every $y\in A(z)\cap Y$ and the real numbers $\lambda_i(y)=l_y(\overline g_i)$ we have $l_y(\overline f)=\sum_{i=1}^{q}\lambda_i(y)\overline f(x_i(z))$ for 
all $\overline f\in E(\overline X)$. Let's show this is true for all $y\in A(z)$. So, fix
$y\in A(z)\backslash Y$ and take
a sequence $\{y_m\}\subset Y\cap F_n(p,q)$ converging to $y$. Then, by Claim 12,
$S_{p,q}(y_m)=\{x_i(y_m)\}_{i=1}^q\subset X$ and there are real numbers $\{\lambda_i(y_m)\}_{i=1}^q$ such that
$l_{y_m}(\overline f)=\sum_{i=1}^q\lambda_i(y_m)\overline f(x_i(y_m))$ for all $\overline f\in E(\overline X)$. 
On the other hand, since the map $S_{p,q}$ is continuous, each sequence $\{x_i(y_m)\}_{m=1}^\infty$ converges to $x_i(y)=x_i(z)$, $i=1,..,q$. So, we can assume that $\{x_i(y_m)\}_{m=1}^\infty\subset O_i$. Consequently, $l_{y_m}(\overline g_i)=\lambda_i(y_m)$ and $\lim_m\lambda_i(y_m)=l_y(\overline g_i)$. 
Since $\{y_m\}\subset Y\cap Y_{p,q}$, $\sum_{i=1}^q|\lambda_i(y_m)|\leq p$ for each $m$ (see Claim 12). Hence, $\sum_{i=1}^q|l_y(\overline g_i)|\leq p$.
Denoting $\lambda_i(y)=l_y(\overline g_i)$, we obtain $\sum_{i=1}^q|\lambda_i(y)|\leq p$. The last inequality means that all $\lambda_i(y)$ are real numbers. 
Since $\lim_m\overline f(x_i(y_m))=\overline f(x_i(y))$ for all $\overline f\in E(\overline X)$ and each $\overline f(x_i(y))$ is a real number (recall that $z\in Y$, so by Claim 12, $x_i(z)=x_i(y)\in X$), we have $l_y(\overline f)=\lim l_{y_m}(\overline f)=\sum_{i=1}^q\lambda_i(y)\overline f(x_i(y))$.

Finally, the equality $\lambda_i(y)=\overline{\varphi(g_i)}(y)$ implies that $\lambda_i$ are continuous on $A(z)$.
\begin{claim}
Let $A(z)=\{y\in F_n(p,q):S_{p,q}(y)=S_{p,q}(z)\}$ with $z\in Y\cap F_n(p,q)$. Then there is a linear continuous map
$\varphi_z:C_p(S_{p,q}(z))\to C_p(A(z))$ such that $\varphi_z(C(S_{p,q}(z)))$ is dense in  $C_p(A(z))$.
\end{claim}
Following the previous notations, for every $h\in C(S_{p,q}(z))$ and $y\in A(z)$ we define $\varphi_z(h)(y)=\sum_{i=1}^q\lambda_i(y)h(x_i(z))$. Because $\lambda_i$ are continuous real-valued functions on $A(z)$, so is each $\varphi_z(h)$. Continuity of $\varphi_z$ with respect to the pointwise convergence topology is obvious. Let's show that $\varphi_z(C(S_{p,q}(z)))$ is dense in  $C_p(A(z))$. Indeed, take $\theta\in C_p(A(z))$ and its neighborhood $V\subset C_p(A(z))$. Then extend $\theta$ to a function $\overline\theta\in C(\overline Y)$.   
Because the set of real-valued elements of $E_p(\overline Y)$ is dense in $C_p(\overline Y)$, there is $\overline g\in E_p(\overline Y)$ with $\overline g|A(z)\in V$, so $\overline g(y)\in\mathbb R$ for all $y\in\overline Y$. 
Next, choose $f\in E(X)$ such that $\varphi(f)=g$.  Since $z\in Y$, each $x_i(z)\in X$. So, all $\overline f(x_i(z))$ are real numbers. Then $h=\overline f|S_{p,q}(z)\in C(S_{p,q}(z))$ and, according to Claim 14, we have $\varphi_z(h)=\overline g|A(z)$.

\begin{claim}
The fibers $A(z)$ of the map $S_{p,q}: F_n(p,q)\to [\overline X]^q$ are $0$-dimensional for all $z\in Y\cap F_n(p,q)$. 
\end{claim}
Since $|S_{p,q}(z)|=q$, $C_p(S_{p,q}(z))$ is isomorphic to $\mathbb R^q$. Hence, Claim 15 together with basic facts from linear algebra imply that $|A(z)|\leq q$, in particular $A(z)$ is zero-dimensional.

Now, we can complete the proof of Proposition 4.1. As in the proof of Proposition 2.1, using that $\overline X^q\in\mathcal P$, we can show that 
$[\overline X]^q\in\mathcal P$.  
Therefore, condition $(c')$ implies that each $F_n(p,q)$ has also the property $\mathcal P$, and so does $Y_n(p,q)=F_n(p,q)\cap Y$. Finally, since $Y$ is the union of its closed subsets $Y_n(p,q)$, we conclude that $Y\in\mathcal P$.
\end{proof}
\begin{lem}
For every countable set $\Phi'\subset C(\beta Y,\overline{\mathbb R})$ there is a countable set $\Phi\subset C(\beta Y,\overline{\mathbb R})$ containing $\Phi'$ such that $(\triangle\Phi)(\beta Y)$ is homeomorphic to $(\triangle\Phi')(\beta Y)$ and the set of real-valued elements of $E_\Phi=\{\pi_{\overline g}:\overline g\in\Phi\}$ is dense in
$C_p((\triangle\Phi)(\beta Y))$.
\end{lem}
\begin{proof}
Let $\phi'=\triangle\Phi'$. Since $\Phi'$ is countable, $\phi'(\beta Y)$ is a metrizable compactum. Hence, by \cite[Proposition 1.2]{gu}, there is a countable $QS$-algebra $E\subset C(\phi'(\beta Y))$. Let $\Phi=\Phi'\cup\{g\circ\phi':g\in E\}$. Since the functions of $E$ separate the points and closed subsets of 
$\phi'(\beta Y)$, $(\triangle\Phi)(\beta Y)$ is homeomorphic to $\phi'(\beta Y)$. Since $E$ is a $QS$-algebra on $\phi'(\beta Y)$, $E$ is a dense subset of $C_p(\phi'(\beta Y))$. Clearly $E$ is a subset of $E_\Phi$ and consists of real-valued functions.
\end{proof}
\begin{lem}
Let $X$ be a $0$-dimensional space and $\Psi'\subset C(X)$ be a countable set. Then there is a countable admissible set $\Psi\subset C(X)$ containing $\Psi'$
and a $0$-dimensional metrizable compactification $\overline X_\Psi$ of $X_\Psi=(\triangle\Psi)(X)$ having a countable base $\mathcal B$ such that:
\begin{itemize}
\item $\overline X_\Psi=(\triangle\overline\Psi)(\beta X)$ with $\overline\Psi=\{\overline f:f\in\Psi\}\subset C(\beta X,\overline{\mathbb R})$;
\item Each $\pi_f$, $f\in\Psi$, is extendable to a map $\overline\pi_f:\overline X_\Psi\to\overline{\mathbb R}$; 
\item $E(X_\Psi)=\{\pi_f:f\in\Psi\}$ is a countable $QS$-algebra on $X_\Psi$ and $E(\overline X_\Psi)=\{\overline\pi_f:f\in\Psi\}$ contains a countable $QS$-algebra on $\overline X_\Psi$ satisfying condition $(2.3)$;
\item For every finite open cover $\gamma$ of $\overline X_\Psi$ with elements from $\mathcal B$ the family $E(\overline X_\Psi)$ contains a partition of unity subordinated to $\gamma$.
\end{itemize}
\end{lem}
\begin{proof}
We first choose a countable admissible set $\Psi_0\subset C(X)$ such that $\dim(\triangle\Psi_0)(X)=0$ and $\Psi'\subset \Psi_0$, see $(3.2)-(3.3)$. Then, by Lemma 3.2 there is a metrizable compactification $Z_0$ of $X_0=(\triangle\Psi_0)(X)$ with $\dim Z_0=0$ such that each $\pi_f$, $f\in\Psi_0$, is extendable to a map 
$\overline\pi_f:Z_0\to\overline{\mathbb R}$. Choose a countable $QS$-algebra $C_0\subset C(Z_0)$ satisfying condition $(2.3)$ (that can be done because $Z_0$ has a countable base, see the explanations in $(2.3)$). 
Let $\mathcal B_0$ be a countable base for $Z_0$ and for every finite open cover $\gamma$ of $Z_0$ consisting of open sets from $\mathcal B_0$ fix a partition of unity $\alpha_\gamma$ subordinated to $\gamma$. Because the family $\Omega_0$ of all finite open covers of $Z_0$ consisting of elements of $\mathcal B_0$ is countable, so is the family $E(Z_0)=\{\overline f:f\in\Psi_0\}\cup\{\alpha_\gamma:\gamma\in\Omega_0\}\cup C_0$. The set $E_0=\{h|X_0:h\in E(Z_0)\}$ may not be a $QS$-algebra on $X_0$ but, according to \cite[Proposition 1.2]{gu}, there exists a countable $QS$-algebra $\Theta_1$ on $X_0$ containing $E_0$ as a proper subset. For every $h\in\Theta_{1}$ the function $h\circ(\triangle\Psi_0):X\to{\mathbb R}$ can be extended to a map $\widetilde h:\beta X\to\overline{\mathbb R}$ and let $P_h=\widetilde h(\beta X)$.
Because $\{h|X_0:h\in C_0\}\subset\Theta_1$ and it separates the points and the closed sets of $X_0$, by Lemma 3.2, there is a metrizable compactification $Z_1$ of $X_0$ such that $\dim Z_1=0$, each $h\in\Theta_1$ is extendable to a map $\overline h:Z_1\to\overline{\mathbb R}$. 
The compactification $Z_1$ is a closed subset of the product $Z_0\times\prod_{h\in\Theta_1}P_h$. Then the projection  
$Z_0\times\prod_{h\in\Theta_1}P_h\to Z_0$ provides a map 
$\theta^1_0:Z_1\to Z_0$ such that $\theta^1_0\circ j_1=j_0$ with $j_1:X_0\to Z_1$ and $j_0:X_0\to Z_0$ being the corresponding embeddings.  

Next, fix a base $\mathcal B_1$ on $Z_1$ containing all sets $(\theta^1_0)^{-1}(U)$, $U\in\mathcal B_0$, which is closed under finite intersections, and consider the family $\Omega_1$ of all finite open covers of $Z_1$ consisting of sets from $\mathcal B_1$.
For each $\gamma\in\Omega_1$ fix a partition of unity $\alpha_\gamma$ subordinated to $\gamma$ and let $E(Z_1)=\{\overline h:h\in\Theta_1\}\cup\{\alpha_\gamma:\gamma\in\Omega_1\}\cup C_1$ and $E_1=\{f|X_0:f\in E(Z_1)\}$, where $C_1\subset C(Z_1)$ is  a countable $QS$-algebra of $Z_1$ satisfying condition $(2.3)$ such that $\{h\circ\theta^1_0:h\in C_0\}\subset C_1$. 
We construct by induction an increasing sequence $\{\Theta_n\}$ of countable $QS$-algebras on $X_0$, $0$-dimensional metrizable compactifications $Z_n$ of $X_0$ with a countable base $\mathcal B_n$, countable $QS$-algebras $C_n\subset C(Z_n)$ on $Z_n$, countable families 
$E(Z_n)\subset C(Z_n,\overline{\mathbb R})$ and surjective maps $\theta^{n+1}_n:Z_{n+1}\to Z_{n}$ such that:
\begin{itemize}
\item[(1)] $\mathcal B_{n+1}$ contains all inverse images $(\theta^{n+1}_n)^{-1}(U)$, $U\in\mathcal B_{n}$, and is closed under finite intersections;
\item[(2)] $C_{n+1}$ satisfies condition $(2.3)$ and $\{h\circ\theta^{n+1}_n:h\in C_{n}\}\subset C_{n+1}$;
\item[(3)] Every $Z_{n+1}$ is a 0-dimensional metrizable compactification of $X_0$ such that every $h\in\Theta_{n+1}$ is extendable to a map $\overline h\in C(Z_{n+1},\overline{\mathbb R})$;
\item[(4)] $E(Z_{n+1})=\{\overline h:h\in\Theta_{n+1}\}\cup\{\alpha_\gamma:\gamma\in\Omega_{n+1}\}\cup C_{n+1}$, where $\Omega_{n+1}$ is the family of all finite open covers $\gamma$ of $Z_{n+1}$ with elements from $\mathcal B_{n+1}$ and $\alpha_\gamma$ is a partition of unity subordinated to $\gamma$;
\item[(5)] $E_n=\{f|X_0:f\in E(Z_n)\}\subset\Theta_{n+1}$.        
\end{itemize}
If the construction is performed for all $k\leq n$, let $E_n=\{h|X_0:\overline h\in E(Z_n)\}$ and $\Theta_{n+1}$ be a countable $QS$-algebra on $X_0$ with $E_n\subset\Theta_{n+1}$. For every $h\in\Theta_{n+1}$ the function $h\circ(\triangle\Psi_0):X\to{\mathbb R}$ can be extended to a map $\widetilde h:\beta X\to\overline{\mathbb R}$ and let $P_h=\widetilde h(\beta X)$.
By Lemma 3.2, there exists a metrizable compactification $Z_{n+1}$ of $X_0$ such that $\dim Z_{n+1}=0$ and each $h\in\Theta_{n+1}$ is extendable to a map $\overline h:Z_{n+1}\to\overline{\mathbb R}$.  The space $Z_{n+1}$ is a closed subset of 
$Z_n\times\prod_{h\in\Theta_{n+1}}P_h$. So, the projection $Z_n\times\prod_{h\in\Theta_{n+1}}P_h\to Z_n$ determines a map $\theta^{n+1}_n:Z_{n+1}\to Z_n$ such that  $\theta^{n+1}_{n}\circ j_{n+1}=j_{n}$, where $j_{n+1}:X_0\to Z_{n+1}$ and $j_{n}:X_0\to Z_{n}$ are the corresponding embeddings. Next, choose a base $\mathcal B_{n+1}$ of $Z_{n+1}$, a countable $QS$-algebra $C_{n+1}\subset C(Z_{n+1})$ and a countable family $E(Z_{n+1})$ satisfying condition $(1)-(5)$.

Since $\{\Theta_n\}$ is an increasing sequence of $QS$-algebras on $X_0$, $\Theta=\bigcup_n\Theta_n$ is also a countable $QS$-algebra on $X_0$, and the limit space $Z$ of the inverse sequence $S=\{Z_n,\theta^n_{n-1}\}$ is  $0$-dimensional. Because $\theta^{n+1}_{n}\circ j_{n+1}=j_{n}$ with 
$j_{n+1}:X_0\to Z_{n+1}$ and $j_{n}:X_0\to Z_{n}$ being embeddings, there is an embedding $j:X_0\to Z$ such that $\theta_n\circ j=j_n$, where 
$\theta_n:Z\to Z_n$ is the projection in $S$. So, $Z$ is a compactification of $X_0$ and 
every $h\in\Theta$ is extendable to a map $\overline{\overline h}:Z\to\overline{\mathbb R}$. Indeed, if $h\in\Theta_n$, then $h$ can be extended to a map $\overline h:Z_n\to\overline{\mathbb R}$ and 
 $\overline{\overline h}=\overline h\circ\theta_n$ is an extension of $h$ over $Z$.
Hence, $\Theta=\{g|Z:g\in E(Z)\}$, where $E(Z)=\bigcup_n\{h\circ\theta_n:h\in E(Z_n)\}$. 

Denote $C_n'=\{h\circ\theta_n:h\in C_n\}$, $n\geq 0$. Because $\{h\circ\theta^n_{n-1}:h\in C_{n-1}\}\subset C_n$,
the sequence $\{C'_n\}$ is increasing and $C=\bigcup_n C'_n$ is closed under multiplications, additions and additions by rational numbers. To show it
a $QS$-algebra on $Z$ we need to prove that for every point $z\in Z$ and its neighborhood $U\subset Z$ there is $h\in C$ such that $h(z)=1$ and $h(Z\backslash U)=0$. But that follows from condition $(2.3)$, so let show that $C$ satisfies condition $(2.3)$. To this end, take a compact set $K\subset Z$ and an open set $W\subset Z$ containing $K$. Then there exist $n$, a compact set $K_n\subset Z_n$ and open set $W_n\subset Z_n$ containing $K_n$ such that $K\subset\theta^{-1}(K_n)\subset\theta^{-1}(W_n)\subset W$. Since $C_n$ satisfies condition $(2.3)$, there is $h\in C_n$
with $h|K_n=1$, $h|(Z_n\backslash W_n)=0$ and $h(z)\in [0,1]$ for all $z\in Z_n$. Then 
$h'=\theta_n\circ h\in C$, $h':Z\to [0,1]$, $h'|K=1$ and $h'|(X\backslash W)=0$.  

Recall that all finite intersections $U=\bigcap_{i=1}^k\theta_i^{-1}(U_i)$ with $U_i\in\mathcal B_i$ form a base $\mathcal B$ for $Z$. Because of the choice of all $\mathcal B_i$, see condition $(1)$, $\mathcal B$ consists of all sets of the form $U=\theta_n^{-1}(U_n)$ with $U_n\in\mathcal B_n$, $n\in\mathbb N$. 
Let's show that for any finite open cover $\gamma$ of $Z$ consisting of sets from $\mathcal B$, the set $E(Z)$ contains a partition of unity subordinated to $\gamma$. Indeed, for any such a cover $\gamma=\{U_1,..,U_k\}$ there is $n$ and a cover $\gamma_n=\{U_1^n,..,U_k^n\}\in\Omega_n$ such that $U_i=\theta_n^{-1}(U_i^n)$. So, there is a partition of unity 
$\alpha_{\gamma_n}=\{h_i^n:i=1,..,k\}$ subordinated to $\gamma_n$ with $\alpha_{\gamma_n}\subset E(Z_n)$.  
Then $\{h^n_i\circ\theta_n:i=1,...,k\}$ is a partition of unity subordinated to $\gamma$ and it is contained in $E(Z)$.
Finally, let
$\Psi=\{h\circ\triangle\Psi_0:h\in\Theta\}$, $\overline\Psi=\{\overline f:f\in\Psi\}\subset C(\beta X,\overline{\mathbb R})$ and $\overline X_\Psi=(\triangle\overline\Psi)(\beta X)$. Since $C$ is a $QS$-algebra on $Z$, it separates the points and the closed sets in $Z$. Moreover, $C\subset E(Z)$, 
which means that
$\overline X_\Psi$ is homeomorphic to $Z$.
\end{proof}
\textit{Proof of Theorem $1.4$.} Let $T:C_p(X)\to C_p(Y)$ be a continuous linear surjection and $\dim X=0$. 
To show that $\dim Y=0$, it suffices to prove that for every $h\in\mathcal F_{Y}$ there exists  $h_0\in\mathcal F_{Y}$ with $h_0\succ h$ and $\dim h_0(Y)=0$.
To this end, fix $h\in\mathcal F_{Y}$. Following the proof of Theorem 1.1 and using Lemmas 4.2-4.3, we are constructing two increasing sequences of countable sets 
$\{\overline\Psi_n\}\subset C(\beta X,\overline{\mathbb R})$, $\{\overline\Phi_n\}\subset C(\beta Y,\overline{\mathbb R})$, metrizable compactifications $\overline X_n=(\triangle\overline\Psi_n)(\beta X)$ and $\overline Y_n=(\triangle\overline\Phi_n)(\beta Y)$ of
the spaces $X_n=(\triangle\Psi_n)(X)$ and $Y_n=(\triangle\Phi_n)(Y)$, where $\Psi_n=\overline\Psi_n|X$ and $\Phi_n=\overline\Phi_n|Y$, 
countable bases $\mathcal B_n'$ and $\mathcal B_n''$ for $\overline X_n$ and 
$\overline Y_n$ and continuous surjections $\theta^{n+1}_n:\overline X_{n+1}\to\overline X_n$, $\delta^{n+1}_n:\overline Y_{n+1}\to\overline Y_n$
satisfying the following conditions (everywhere below, if $f\in C(X)$ then $\overline f:\beta X\to\overline{\mathbb R}$ denotes its extension):
\begin{itemize}
\item[(4.1)] $\triangle\Phi_1\succ h$, $\Phi_n\subset\{T(f):f\in\Psi_{n}\}\subset\Phi_{n+1}$ and $\Psi_n\subset\Psi_{n+1}$;
\item[(4.2)] Each  $\pi_f$, $f\in\Psi_n$, is extendable to a map $\overline\pi_f:\overline X_n\to\overline{\mathbb R}$;
\item[(4.3)] $\Psi_n$ is admissible, $\dim\overline X_n=0$ and $\mathcal B_n'$ satisfies condition $(1)$ from Lemma 4.3;
\item[(4.4)] $E(\overline X_n)=\{\overline\pi_f:f\in\Psi_n\}$ contains a countable $QS$-algebra $C_n\subset C(\overline X_n)$ on $\overline X_n$ such that $C_n$ satisfies condition $(2)$ from the proof of Lemma 4.3;
\item[(4.5)] For every finite open cover $\gamma$ of $\overline X_n$, consisting of sets from $\mathcal B_n'$, there exists a partition of unity $\alpha_\gamma$ subordinated to $\gamma$ with $\alpha_\gamma\subset E(\overline X_n)$;
\item[(4.6)] Every  $\pi_g$, $g\in\Phi_n$, is extendable to a map $\overline\pi_g:\overline Y_n\to\overline{\mathbb R}$;
\item[(4.7)]  $\mathcal B_{n}''$ contains all inverse images $(\delta^{n}_{n-1})^{-1}(U)$, $U\in\mathcal B_{n-1}''$, and is closed under finite intersections;
\item[(4.8)] The set of real-valued functions from $E(\overline Y_n)=\{\overline\pi_g:g\in\Phi_n\}$ is dense in $C_p(\overline Y_n)$.
\end{itemize}
Since $h(Y)$ is a separable metrizable space, there is a countable set $\Phi_1'\subset C(Y)$ with $h=\triangle\Phi_1'$. Let 
$\overline{\Phi_1'}=\{\overline g:g\in\Phi'_1\}\subset C(\beta Y,\overline{\mathbb R})$. By Lemma 4.2, there is a countable set $\overline\Phi_1\subset C(\beta Y,\overline{\mathbb R})$ containing $\overline{\Phi_1'}$ such that $(\triangle\overline\Phi_1)(\beta Y)$ is homeomorphic to $(\triangle\overline{\Phi_1'})(\beta Y)$ and $\{\pi_{\overline g}:\overline g\in\overline\Phi_1\}$ contains a dense subset of $C_p(\overline Y_1)$, where 
$\overline Y_1=(\triangle\overline\Phi_1)(\beta Y)$. Let $\Phi_1=\{g:\overline g\in\overline\Phi_1\}$, $Y_1=(\triangle\Phi_1)(Y)$ and 
$E(\overline Y_1)=\{\overline\pi_{g}:g\in\Phi_1\}$. So, $\Phi_1$ satisfies conditions $(4.6)-(4.7)$.
Next, choose a countable set $\Psi_1'\subset C(X)$ with $T(\Psi_1')=\Phi_1$ and apply Lemma 4.3 to find  a countable admissible set $\Psi_1$ containing 
$\Psi_1'$ and a metrizable compactification $\overline X_1$ of $X_1=\triangle\Psi_1(X)$ satisfying conditions $(4.2)-(4.5)$.

Suppose the construction is done for all $k\leq n$. Let $\Phi_{n+1}'\subset C(Y)$ be a countable set containing $T(\Psi_n)$ and denote 
 $\overline{\Phi'}_{n+1}=\{\overline g:g\in\Phi'_{n+1}\}\subset C(\beta Y,\overline{\mathbb R})$. By Lemma 4.2, there is a countable set $\overline\Phi_{n+1}\subset C(\beta Y,\overline{\mathbb R})$ containing $\overline{\Phi'}_{n+1}$ such that $(\triangle\overline\Phi_{n+1})(\beta Y)$ is homeomorphic to $(\triangle\overline{\Phi'}_{n+1})(\beta Y)$ and $\{\pi_{\overline g}:\overline g\in\overline\Phi_{n+1}\}$ contains a dense subset of $C_p(\overline Y_{n+1})$, where 
$\overline Y_{n+1}=(\triangle\overline\Phi_{n+1})(\beta Y)$. Let $\Phi_{n+1}=\{g:\overline g\in\overline\Phi_{n+1}\}$ and $Y_{n+1}=(\triangle\Phi_{n+1})(Y)$.
Note that $\Phi_n\subset\Phi_{n+1}$ because $\Phi_n\subset T(\Psi_n)$.
Next, choose a countable set $\Psi_{n+1}'\subset C(X)$ with $T(\Psi_{n+1}')=\Phi_{n+1}$ and apply Lemma 4.3 to find  a countable admissible set $\Psi_{n+1}$ containing 
$\Psi_{n+1}'\cup\Psi_n$ and a metrizable compactification $\overline X_{n+1}$ of $X_{n+1}=(\triangle\Psi_{n+1})(X)$ satisfying conditions $(4.1)-(4.5)$.
Because $\overline\Psi_n\subset\overline\Psi_{n+1}$, $\triangle\overline\Psi_{n+1}\succ\triangle\overline\Psi_{n}$. Hence, there exists a map  
$\theta^{n+1}_n:\overline X_{n+1}\to\overline X_n$ defined by $\theta^{n+1}_n=\triangle\overline\Psi_{n}((\triangle\overline\Psi_{n+1})^{-1}(x))$. 
This completes the induction. 

As in the proof of Theorem 1.1, we denote $\overline X_0=(\triangle\overline\Psi)(\beta X)$,
$\overline Y_0=(\triangle\overline\Phi)(\beta Y)$ and $h_0=\triangle\Phi$, where $\Psi=\bigcup_n\Psi_n$ and $\Phi=\bigcup_n\Phi_n$. Clearly, $\Phi=\{T(f): f\in\Psi\}$.
Since $\overline X_0$ is the limit space of the inverse sequence $S_X=\{\overline X_n,\theta^{n+1}_n\}$ and $\dim\overline X_n=0$ for all $n$, $\dim\overline X_0=0$. Moreover, $E(X_0)=\{\pi_f:f\in\Psi\}$ is a countable $QS$-algebra on $X_0$ such that every $\pi_f$ is extendable to a continuous map $\overline\pi_f:\overline X_0\to\overline{\mathbb R}$. Denote $E(\overline X_0)=\{\overline\pi_f:f\in\Psi\}$.
Let also $C=\bigcup_n\{h\circ\theta_n:h\in C_n\}\subset E(\overline X_0)$. The same arguments as in the proof of Lemma 4.3 show that $C$ is a $QS$-algebra on $\overline X_0$ satisfying condition $(2.3)$. Let $\mathcal B'$ be the base of $\overline X_0$ generated by the bases $\mathcal B_n'$. The arguments from the proof of Lemma 4.3 also provide that 
 for every finite open cover $\gamma$ of $\overline X_0$, consisting of open sets from $\mathcal B'$ there exists a partition of unity $\alpha_\gamma$ subordinated to $\gamma$ with $\alpha_\gamma\subset E(\overline X_0)$. 

Because $\Phi_n\subset\Phi_{n+1}$, $\overline\Phi_n\subset\overline\Phi_{n+1}$. This implies $\triangle\Phi_{n+1}\succ\triangle\Phi_{n}$. So, for every $n$ there is map $\delta^{n+1}_n:\overline Y_{n+1}\to\overline Y_n$ defined by $\delta^{n+1}_n(y)=\triangle\Phi_{n}((\triangle\Phi_{n+1})^{-1}(y))$.
Then $\overline Y_0$ is the limit of the inverse sequence $S_Y=\{\overline Y_n,\delta^{n+1}_n\}$. 
Let $E(Y_0)=\{\pi_g:g\in\Phi\}$ and $E(\overline Y_0)=\{\overline\pi_{g}:g\in\Phi\}$. 
We claim that the set of real-valued elements of $E_p(\overline Y_0)$ is dense in $C_p(\overline Y_0)$. 
Indeed, every projections $\delta_n:\overline Y_0\to\overline Y_n$ induces a continuous map 
$\delta_n^*:C_p(\overline Y_n)\to C_p(\overline Y_0)$ defined by $\delta_n^*(h)=h\circ\delta_n$. 
Because the set of real-valued functions from  
$E(\overline Y_n)=\{\overline\pi_{g}:g\in\Phi_n\}$ is dense in $C_p(\overline Y_n)$ and 
$E_p(\overline Y_0)=\bigcup_n\delta_n^*(E(\overline Y_n))$, it suffices to show that  
$\bigcup_n\delta_n^*(C_p(\overline Y_n))$ is dense in $C_p(\overline Y_0)$. To this end, let 
$O=\{f\in C_p(\overline Y_0):|f(\overline y_i)-f_0(\overline y_i)|<\varepsilon_i, i=1,..,k\}$ be a neighborhood of some $f_0\in C_p(\overline Y_0)$, where 
all $\overline y_i$ are different points from $\overline Y_0$. Since the base $\mathcal B''$ of $\overline Y_0$ consists of all sets of the form $\delta_n^{-1}(U_n)$, $U_n\in\mathcal B_n''$, there is $n$ and different points $y_i\in\overline Y_n$ such that $\delta_n(\overline y_i)=y_i$ and $\delta_n^{-1}(U_i)\subset f_0^{-1}(V_i)$, where $U_i\in\mathcal B_n''$ and $V_i$ is the open interval $(f_0(\overline y_i)-\varepsilon_i,f_0(\overline y_i)+\varepsilon_i)$. Then the set $W=\{h\in C_p(\overline Y_n):h(y_i)\in V_i, i=1,2,..,k\}$  is open in $C_p(\overline Y_n)$. So, it contains a function
$h_0\in E(\overline Y_n)$. Hence, $h_0\circ\delta_n$ is a function from $\bigcup_n\delta_n^*(C_p(\overline Y_n))\cap O$. So, the set of all real-valued functions from $E(\overline Y_0)$ is dense in $C_p(\overline Y_0)$.

Because $T$ is linear, so is the map $\varphi:E_p(X_0)\to E_p(Y_0)$ defined by  $\varphi(\pi_f)=\pi_{T(f)}$. 
The arguments from the proof of Theorem 1.1 show that $\varphi$ is continuous, but we are going to prove the more general fact that $\varphi$ has a continuous extension over the linear hull $LE_p(X_0)$. For every $f\in LE(X_0)$ let $f^*=f\circ(\triangle\Psi)\in C(X)$. Evidently, if  
$f=\sum_{i=1}^k\lambda_i\cdot f_i\in LE(X_0)$, then $f^*=\sum_{i=1}^k\lambda_i\cdot f_i^*$. So, we have another description of the map $\varphi$: 
$\varphi(f)=\pi_{T(f^*)}$, $f\in E(X_0)$.
\begin{claim}
Let $f=\sum_{i=1}^k\lambda_i\cdot f_i\in LE(X_0)$ with $f_i\in E(X_0)$ for all $i$. If $y^*\in Y$ and $(\triangle\Phi)(y^*)=y$, then
$T(f^*)(y^*)=\sum_{i=1}^k\lambda_i\cdot T(f_i^*)(y^*)=\sum_{i=1}^k\lambda_i\cdot\varphi(f_i)(y)$. 
\end{claim}
It suffices to show that $T(f^*)(y^*)=\varphi(f)(y)$ for all $f\in E(X_0)$. And that is true because $f^*\in\Psi$, so $T(f^*)\in\Phi$ and $T(f^*)(y^*)=\varphi(f)(y)$. 

We define $\widetilde\varphi:LE(X_0)\to LE(Y_0)$ by 
$\widetilde\varphi(\sum_{i=1}^k\lambda_i\cdot f_i)=\sum_{i=1}^k\lambda_i\cdot\varphi(f_i)$, where $\lambda_i\in\mathbb R$ and $f_i\in E(X_0)$. The continuity of $\widetilde\varphi$ with respect to the pointwise topology is equivalent of the continuity of all linear functionals $\widetilde l_y:LE_p(X_0)\to \mathbb R$ defined by $\widetilde l_y(f)=\widetilde\varphi(f)(y)$, $y\in Y_0$.
So, fix $y_0\in Y_0$ and $f_0=\sum_{i=1}^k\lambda_i\cdot f_i\in LE(X_0)$ with $f_i\in E(X_0)$ such that $\widetilde l_{y_0}(f_0)\in V$ for some open interval $V\subset\mathbb R$. Then $f_0^*=\sum_{i=1}^k\lambda_i\cdot f_i^*$ and, by Claim 17 we have
$$T(f_0^*)(y_0^*)=\sum_{i=1}^k\lambda_i\cdot T(f_i^*)(y_0^*)=\sum_{i=1}^k\lambda_i\cdot\varphi(f_i)(y_0)=\widetilde l_{y_0}(f_0),$$
where $y_0^*\in Y$ with $(\triangle\Phi)(y_0^*)=y_0$.
Consequently, there is a neighborhood $W^*=\{g\in C_p(X):|g(x_j^*)-f_0^*(x_j^*)|<\eta_j, j=1,2,..,m\}$ of $f_0^*$ in $C_p(X)$ such that $T(g)(y_0^*)\in V$ for all $g\in W^*$. Observe that 
$$f_0^*(x_j^*)=\sum_{i=1}^k\lambda_i\cdot f_i^*(x_j^*)=\sum_{i=1}^k\lambda_i\cdot f_i(x_j)=f_0(x_j),$$
where $x_j=(\triangle\Psi)(x_j^*)$. So, $W=\{f\in LE(X_0):|f(x_j)-f_0(x_j)|<\eta_j,j=1,2,..,m\}$ is a neighborhood of $f_0$ in $LE_p(X_0)$. If $g\in W$, 
then $g=\sum_{s=1}^m\lambda_s\cdot g_s$ for some $g_s\in E(X_0)$. Hence, $g^*=\sum_{s=1}^m\lambda_s\cdot g_s^*\in W^*$ which means that
$T(g^*)(y_0^*)\in V$. Finally, according to Claim 17, $T(g^*)(y_0^*)=\sum_{s=1}^m\lambda_s\cdot\varphi(g_s)(y_0)=\widetilde l_{y_0}(g)\in V$.
Thus, $\widetilde\varphi:LE_p(X_0)\to LE_p(Y_0)$ is continuous.

Therefore, the spaces $X_0$, $\overline X_0$, $Y_0$ and $\overline Y_0$ satisfy the assumptions of Proposition 4.1.
So, we apply Proposition 4.1 with $X=X_0$ and $Y=Y_0$ to conclude that $\dim Y_0=0$. That completes the proof of Theorem 1.4. $\Box$

\smallskip
\textbf{Acknowledgments.} The authors express their
gratitude to referees for their many suggestions and improvements. 

\end{document}